\documentclass[twoside]{amsart}          

\usepackage{
enumerate}
\usepackage[all]{xy}
\usepackage{tikz}     
\usepackage{amssymb,amsthm,amsmath,amscd}   
\usepackage{fancyhdr}                       
\usepackage{dsfont}                         
\usepackage{color}

\theoremstyle{plain} 
\newtheorem{theorem}{Theorem}[section] 
\newtheorem{corollary}[theorem]{Corollary} 
\newtheorem{proposition}[theorem]{Proposition} 
 
\newtheorem{lemma}[theorem]{Lemma} 
 
\theoremstyle{definition} 
\newtheorem{definition}[theorem]{Definition} 
\theoremstyle{remark} 
\newtheorem{remark}[theorem]{Remark} 
\newtheorem{example}[theorem]{Example} 

\def\R{\mathbb{R}} 
\def\N{\mathbb{N}} 
\def\Z{\mathds{Z}} 
 
\def\C{\mathds{C}} 
\def\cA{\mathcal{A}} 
 
\def\cC{\mathcal{C}}

\def\cG{\mathcal{G}}
\def\cH{\mathcal{H}}             
 
\def\cJ{\mathcal{J}}  
\def\cK{\mathcal{K}}

\def\cP{\mathcal{P}} 
\def\cR{\mathcal{R}}

\def\cV{\mathcal{V}} 
\def\cW{\mathcal{W}} 
 
\def\cZ{\mathcal{Z}}

\newcommand\pa{\partial}

\author[Catarina Carvalho]{Catarina Carvalho}
\address{C. Carvalho, Instituto Superior T\'ecnico, Math. Dept., UTL,
Av. Rovisco Pais, 1049-001 Lisbon, Portugal} \email{ccarv@math.ist.utl.pt} 

\author[Yu Qiao]{Yu Qiao}
\address{Y. Qiao, Chern Institute of Mathematics,
        Nankai University, Tianjin 300071 \\
       People's Republic of China} \email{fishqiao@gmail.com}


\date{\today} \title[Layer potentials $C^*$-algebras of
  domains with conical points]{Layer Potentials C*-algebras of domains with conical
  points}

\begin{document} 



\begin{abstract} 
To a domain with conical points $\Omega$, we associate a natural
$C^*$--algebra that is motivated by the study of boundary value
problems on $\Omega$, especially using the method of layer
potentials. In two dimensions, we allow $\Omega$ to be a domain with
ramified cracks. We construct an explicit groupoid associated to $\pa
\Omega$ and use the theory of pseudodifferential operators on
groupoids and its representations to obtain our layer potentials $C^*$-algebra. We
study its structure, compute the associated $K$-groups, and prove Fredholm conditions for the natural pseudodifferential operators affiliated to
this $C^*$-algebra.
\end{abstract} 

\maketitle 

\tableofcontents


\section*{Introduction\label{Introduction}} 

Let $\Omega$ be a bounded domain with conical points in
$\mathbb{R}^n$, $n \ge 2$, that is, $\Omega$ is locally diffeomorphic
to a cone with smooth, possibly disconnected, basis.  To $\Omega$, or
more precisely to $\pa \Omega$, we associate a natural $C^*$--algebra, the \emph{ layer potentials $C^*$-algebra}, 
that is motivated by the study of boundary value problems on $\Omega$,
especially by applications of the method of layer potentials. In two
dimensions, we allow $\Omega$ to be a domain with ramified cracks. The main aim of this paper is to prove Fredholm conditions for the natural pseudodifferential operators affiliated to
the layer potentials $C^*$-algebra. We make use of pseudodifferential calculus on groupoids, so our first step is to associate to $\Omega$ a \emph{boundary groupoid} and study its structure, as well as the structure of the resulting groupoid $C^*$-algebra. Moreover, we compute the associated $C^*$-algebraic $K$-groups and show that they depend only on the number of conical points. We expect our Fredholm criterion to have applications to boundary operators coming from the study of boundary value problems in $\Omega$.

One of the classical approaches to solving boundary problems for
(strongly) elliptic equations is via the method of layer potentials,
which reduces differential equations to  boundary integral
equations. More explicitly, by reduction, we want to invert an
operator of the form ``$\frac{1}{2} +K$" on appropriate boundary
function spaces. For instance, if the boundary is $\cC^2$, then the
relevant operator $K$ is compact \cite{Fol, Kress} on
$L^2(\pa\Omega)$. So the operator $\frac{1}{2}+K$ is Fredholm. Hence
we can apply the classical Fredholm theory to the operator
$\frac{1}{2}+K$ to solve the Dirichlet problem. If the boundary is
$\cC^1$, then the integral operator $K$ on $\pa \Omega$ is still
compact \cite{FJR} on $L^2(\pa\Omega)$, but that no longer holds if
there are singularities on the boundary \cite{Els, FJL, Kon, Kress,
  Lew, LP, IMitrea2, IMitrea3, MitreaNistor}. In the singular case, it
is therefore natural to look for larger $C^*$-algebras containing
these boundary integral operators and where we still have boundedness
and Fredholm criteria.  See, for instance, the survey \cite{Mazya}, where the importance of understanding algebras of pseudodifferential operators on singular spaces is emphasized.  In this paper, we tackle the case of conical domains, possibly with cracks and make some progress in this direction. Our approach is to construct a Lie
groupoid associated to $\pa \Omega$ and take its convolution
$C^*$-algebra, so that tools from analysis on groupoids become available.

We first consider a desingularization $ \Sigma(\Omega)$ of $\Omega$ as
in {\cite{BMNZ, Kon, MelroseAPS}},
 which in this case
basically amounts to replacing a, possibly disconnected, neighborhood
of the conical points by a cylinder. One
obtains in that way a manifold with corners, which moreover has the
structure of a Lie manifold with boundary \cite{ALN, BMNZ}. In particular
there are naturally defined Sobolev spaces behaving much as in the
smooth boundary case, and nice regularity and Fredholmness criteria
\cite{ALN, LMN, LN}.  On the boundary, we have a decomposition
\begin{equation*}
  \pa \Sigma(\Omega) = \pa' \Sigma(\Omega) \cup \pa'' \Sigma(\Omega),
\end{equation*}
where $\pa' \Sigma(\Omega) $ are the hyperfaces that correspond to
faces of $\Omega$ and $\pa'' \Sigma(\Omega) $ are the hyperfaces at
infinity. The space $\pa' \Sigma(\Omega)$ can be identified with a
desingularization of $\pa \Omega$. The layer potentials  $C^*$-algebra is then defined
as the $C^*$-algebra of a suitable Lie groupoid $\cG$ with units $\pa'
\Sigma(\Omega)$ (Definition \ref{def.layerC*}). We use the theory of pseudodifferential operators on
groupoids to identify $C^*(\cG)$ with an algebra of operators.  In
fact, $C^*(\cG)$ is an ideal of the norm closure of the
  algebra of order zero pseudodifferential operators on
$\cG$. There is a bounded, injective representation $\pi$ of the algebra of pseudodifferential operators on $\cG$ on $C_c^\infty(\pa' \Sigma(\Omega))$, mapping $C^*(\cG)$ to the compact operators.
Moreover, the Sobolev spaces $H^m(\pa' \Sigma(\Omega))$, defined using the Lie structure, can be identifed \cite{BMNZ} with weighted Sobolev spaces $ \cK_{a}^m(\pa \Omega)$ on $\pa \Omega$, $a\in \R$ (see (\ref{def.weighted}) for the precise definition).
We show that Fredholm criteria for operators on
groupoids, as in \cite{LMN, LN}, apply in our case, so that we obtain (Theorem \ref{fredholm}) 
that Fredholmness of an operator
$$\pi(P):
\cK^{m}_{\frac{n-1}{2}}(\partial\Omega) \rightarrow \cK^{0}_{\frac{n-1}{2}}(\partial\Omega) $$
is equivalent to ellipticity and invertibility of a
family of operators 
\begin{equation*}
  P_x : H^{m}(\R^+ \times \pa \omega_p)\rightarrow L^2(  \R^+ \times \pa \omega_p)
\end{equation*}
where $p$ is a conical point, $\omega_p$ is the basis of the (local) cone at $p$, and $x$ is at the boundary of $\pa' \Sigma(\Omega)$. We also obtain (Theorem \ref{fredholm2}) that we can replace the second condition above with invertibility of a family of Mellin convolution operators on $\R^+ \times \pa \omega_p$. In two dimensions, we allow our domain to have cracks (Corollary \ref{fredholm.crack}). Fredholm conditions of this form, which are often referred generally as full ellipticity, appear in many contexts related to index theory on singular spaces, see for instance \cite{CordesBVP, EgorovSchulze, MelroseAPS, MelroseScattering, Nistor1, SchroheSG, SchroheSchulze1, SchroheSchulze2, Schulze} (and references therein).

When there are no singularities, the groupoid $\cG$ reduces to the
pair groupoid of $\pa\Omega$ and the $C^*$-algebra $C^*(\cG)$ is then
isomorphic to the algebra of compact operators on $\pa
\Omega$. Consequently, the representation theory of $C^*(\cG)$ -- in
the case when there are no singularities -- can be used to recover
results from Fredholm theory.

In the case of a straight cone with basis $\omega \subset S^{n-1}$,
the desingularization is $\Sigma(\Omega)=[0,\infty)\times
  \overline{\omega}$, and $\pa' \Sigma(\Omega)=[0,\infty)\times \pa
\omega$.  Our construction gives $$\cG= ([0,\infty)\rtimes \R^+)\times
  (\pa \omega)^2,$$ where $[0,\infty)\rtimes \R^+$ denotes the action
groupoid and $ (\pa \omega)^2$ the pair groupoid (see Section
\ref{straightcone} for details).  When the basis $\pa \omega$ of the
cone at a singular point $P$ is disconnected, which is always the case
in two dimensions, for instance, the desingularization will associate
several boundary faces to $P$. We allow here interaction between these
faces at the groupoid level, in that there will be arrows between
different connected components. This will be important in
applications.  

Since $\pa' \Sigma(\Omega)$ is a manifold
with smooth boundary, one can also consider Melrose's $b$-calculus
\cite{MelroseAPS} to obtain a well-behaved class of pseudodifferential
operators on $\pa' \Sigma(\Omega)$. Our pseudodifferential calculus
contains the $b$-pseudodifferential operators, in that the boundary
groupoid defined here contains the $b$-groupoid as an open
subgroupoid. The main difference at the groupoid level is that in the usual $b$-calculus, there is no interaction (no arrows) between the different faces at $P$.

Groupoids and groupoid $C^*$-algebras have appeared useful in the
analysis over singular spaces and, in particular, spaces with conical singularities, see for instance \cite{MonthubertSchrohe, AJ2, DebordLescure1, DebordLescure2, DLN, LMN, LN, Monthubert1, MonthubertNistor, MRen}.

  In \cite{DebordLescure1, DebordLescure2}, Debord and Lescure associated
to a pseudomanifold $M$ with a conical singularity a Lie groupoid
which can be used to prove index theorems \cite{DebordLescure2, DLN}
and to deal with elliptic theory on manifolds with conical
singularities \cite{Lescure}. The general idea is to blow up the
conical point to get a cylinder, then glue the pair groupoid of this
cylinder and the tangent space of the smooth part of $M$ in a way such
that there is a natural smooth structure on the resulting
groupoid. (This groupoid is called `a tangent space of $M$' in
\cite{DebordLescure2}.)

On a different line, the $C^*$-algebra of the transformation groupoid
(Example \ref{expl_transf_grpd}) $\cH:=[0,\infty)\rtimes \R^+$ is the
algebra of Wiener-Hopf operators on $\R^+$. In the work of Muhly and
Renault \cite{MRen}, they used groupoid techniques to identify the
$C^*$-algebra generated by Wiener-Hopf operators defined over
polyhedral cones or homogeneous, self-dual cones, with the
$C^*$-algebra of a locally compact measured groupoid. Their
construction is motivated by the study of the structure of the
$C^*$-algebra generated by multivariable Weiner-Hopf operators from the
groupoid point of view. Moreover, in \cite{AJ2} this groupoid
approach was used to analyze the structure of the (generalized)
Wiener-Hopf $C^*$-algebra and do index theory for Wiener-Hopf
operators on cones.

The construction presented in this paper has different aims, in that
it is motivated by PDE's and comes from the nature of the
singularities. Our general purpose is that certain boundary
convolution integral operators are in fact in the groupoid
$C^*$-algebra, so we need to consider groupoids over the
(desingularized) boundary of the conical domain. Furthermore, from this fact and the results presented here, we will be able to show that these integral operators are Fredholm between suitable weighted Sobolev spaces for domains with conical points of dimension greater than $3$. However, for domains with cracks, the resulting layer potential operators are no longer Fredholm. All these issues will be analyzed thoroughly in a forthcoming paper.

Let us briefly review the contents of each section. In Section
\ref{groupoid}, we review some basic knowledge of Lie groupoids, define pseudodifferential operators on a Lie
groupoid and, from this, we define the $C^*$-algebra
of a Lie groupoid.  In Section \ref{weightedspaces}, we give the main
concepts concerning domains with conical points, including the
desingularization $\Sigma(\Omega)$ and the definition of weighted
Sobolev spaces on $\Omega$ and $\pa\Omega$, which will be the natural domains of our operators.  In Section
\ref{straightcone}, we consider the case of straight
cones. We construct a canonical Lie groupoid over the desingularized boundary (a cylinder), whose $C^*$-algebra coincides with that of the  Toeplitz / Wiener-Hopf operators.
 In Section \ref{GpdConstruction}, we generalize these constructions to the 
case of domains with conical points. We construct explicitly a
canonical Lie groupoid $\cG$ associated to $\pa\Omega$, the so-called boundary groupoid,  and study the
properties of the groupoid $C^*$-algebra $C^*(\cG)$, which we dub the layer potentials $C^*$-algebra. In particular, we show that the boundary groupoid and the layer potentials $C^*$-algebra only depend, up to equivalence,  on the number of singularities of the conical domain. In dimension two, we allow our domains to have ramified cracks; this case is considered in Section \ref{crack}. In Section \ref{Ktheory}, we
compute the $K$-theory of  the layer potentials $C^*$-algebra $C^*(\cG)$ and of the indicial algebra at the boundary.  Lastly, in Section
\ref{Fredholmness}, we obtain a Fredholm criterion for
pseudodifferential operators on $\cG$ (including $C^*(\cG)$).

\vskip 2ex
We would like to thank Victor Nistor for the suggestion of this problem and stimulating thinking. And we are grateful to Anna Mazzucato and John Roe for many useful discussions and suggestions.

\smallskip

\section{
Pseudodifferential Operators on Groupoids and Groupoid
$C^*$-algebras \label{groupoid}}

\subsection{Lie groupoids and Lie algebroids}
In this subsection, we review some basic facts on Lie groupoids and Lie
algebroids. We begin with the definition of groupoids.

\begin{definition}
A \emph{groupoid} is a small category $\mathcal{G}$ in which each arrow is invertible.  
\end{definition}

Let us make this definition more explicit \cite{CannasWeinstein, LN,
  MM, Ren}. A groupoid $\mathcal{G}$ consists of two sets, a set of
objects (or units) $\mathcal{G}_0$ and a set of arrows
$\mathcal{G}_1$. Usually we shall denote the space of units of
$\mathcal{G}$ by $M$ and we shall identify $\mathcal{G}$ with
$\mathcal{G}_1$. Each object of $\mathcal{G}$ can be identified with
an arrow of $\mathcal{G}$. We have an injective map
$u:M:=\mathcal{G}_0\rightarrow \mathcal{G}_1$, where $u(x)$ is the
identity arrow of an object $x$. To each arrow $g\in \mathcal{G}$ we
associate two units: its domain $d(g)$ and its range $r(g)$. The
multiplication $\mu(g,h)=gh$ of two arrows $g,h\in \mathcal{G}$ is not
always defined; it is defined exactly when $d(g)=r(h)$. The
multiplication is associative. The inverse of an arrow is denoted by
$g^{-1}=\iota(g)$.

A groupoid $\mathcal{G}$ is therefore completely determined by the
sets $\mathcal{G}_0$, $\mathcal{G}_1$ and the structural maps
$d,r,\mu,u,\iota$. Consequently, we sometimes denote $\mathcal{G}=(
\mathcal{G}_0, \mathcal{G}_1,d, r, \mu,u,\iota)$. The structural maps
satisfy the following properties:
\begin{enumerate}
\item $d(hg)=d(g)$, $r(hg)=r(h)$,
\item $k(hg)=(kh)g$
\item $u(r(g))g=g=gu(d(g))$, and
\item $d(g^{-1})=r(g)$, $r(g^{-1})=d(g)$, $g^{-1}g=u(d(g))$, and $gg^{-1}=u(r(g))$
\end{enumerate}
for any $k,h,g \in \mathcal{G}_1$ with $d(k)=r(h)$ and
$d(h)=r(g)$. The structural maps in a groupoid $\mathcal{G}$ together
fit into a diagram \cite{MM}
  
\begin{equation*} \xymatrix{ 
 \mathcal{G}_1\times_{\mathcal{G}_0} \mathcal{G}_1 \ar[r]^<<<<<\mu 
  & \mathcal{G}_1 \ar[r]^\iota 
      & \mathcal{G}_1\ar@<2.5pt>[r]^d \ar@<-2.5pt>[r]_r
          & \mathcal{G}_0 \ar[r]^u 
              &\mathcal{G}_1. 
}\end{equation*}

The following definition is taken from \cite{LN}.

\begin{definition}
A \emph{Lie groupoid} is a groupoid 
\begin{equation*} 
  \mathcal{G} = ( \mathcal{G}_0, \mathcal{G}_1,d, r, \mu, u, \iota)
\end{equation*}
such that $M:= \mathcal{G}_0$ and $ \mathcal{G}_1$ are smooth
manifolds (with or without corners), the structural maps $d, r, \mu,
u$, and $\iota$ are smooth, the domain map $d$ is a submersion, and
all the spaces $M$ and $ \mathcal{G}_x=d^{-1}(x)$, $x\in M$, are
Hausdorff.
\end{definition}

We now recall the definition of a Lie algebroid \cite{LN}.

\begin{definition}
A \emph{Lie algebroid} $A$ over a manifold $M$ is a vector bundle $A$
over $M$, together with a Lie algebra structure on the space
$\Gamma(A)$ of the smooth sections of $A$ and a bundle map $\rho:
A\rightarrow TM$, extended to a map between sections of theses
bundles, such that
\begin{enumerate}
\item[(1)] $\rho([X,Y])=[\rho(X),\rho(Y)]$;
\item[(2)] $[X,fY]=f[X,Y]+(\rho(X)f)Y$,
\end{enumerate}
for all smooth sections $X$ and $Y$ of $A$ and any smooth function $f$
on $M$. The map $\rho$ is called the \emph{anchor}. Usually we shall
denote by $(A,\rho)$ such a Lie algebroid.
\end{definition}

Consider a Lie groupoid $\mathcal{G}$ with units $M$. We can associate
a Lie algebroid $A(\cG)$ to $\cG$ as follows \cite{Mac}. The
$d$-vertical subbundle of $T\mathcal{G}$ for $d:\mathcal{G}\rightarrow
M$ is denoted by $T^d(\mathcal{G})$ and called simply the $d$-vertical
bundle for $\mathcal{G}$. It is an involutive distribution on
$\mathcal{G}$ whose leaves are the components of the $d$-fibers of
$\mathcal{G}$. (Here involutive distribution means that
$T^d(\mathcal{G})$ is closed under the Lie bracket, i.e. if $X,Y \in
\mathfrak{X}(\mathcal{G})$ are sections of $T^d(\mathcal{G})$, then
the vector field $[X,Y]$ is also a section $T^d(\mathcal{G})$.) Hence
we obtain
\begin{equation*}
  T^d\mathcal{G}=\text{ker} \ d_*=\displaystyle{\bigcup_{x\in
      M}T\mathcal{G}_x}\subset T\mathcal{G}.
\end{equation*}
The \emph{Lie algebroid} of $\cG$, denoted by $A(\mathcal{G})$, is
defined to be $T^d(\mathcal{G})|_M$, the restriction of the
$d$-vertical tangent bundle to the set of units $M$. In this case, we
say that $\cG$ integrates $A(\cG)$.

Let $f_1, f_2\in \cC^\infty_c(\cG)$ and fix a Haar system of measures
$d\mu_x$ on the $d$-fibers. The \emph{convolution} product of $f_1$
and $f_2$ is defined as
$$f_1*f_2(h):=\int_{d^{-1}(d(h))}f_1(g)f_2(g^{-1}h)d\mu_{d(h)},$$ and
$\cC^\infty_c(\cG)$ becomes a $*$-algebra with $f^*(g):=
\overline{f(g^{-1})}$. Taking the closure with respect to a suitable
norm, given by the $\text{sup}$ over all bounded representations, we
get the \emph{groupoid $C^*$-algebra} $C^*(\cG)$ (see \cite{Ren}). For
our purposes, it is preferable to define $C^*(\cG)$ as an algebra of
operators on $\cG$, as we shall see now (Definition \ref{def.C*} below).

\subsection{Pseudodifferential operators  and groupoid $C^*$-algebras}
We recall here the construction of the space of pseudodifferential
operators associated to a Lie groupoid $\mathcal{G}$ with units $M$
\cite{LMN, LN, Monthubert, Monthubert1, MonthubertPierrot, NWX}. The
dimension of $M$ is $n>0$.

Namely, let $P=(P_x)$, $x\in M$ be a smooth family of
pseudodifferential operators acting on $\mathcal{G}_x$. We say that
$P$ is \emph{right invariant} if $P_{r(g)}U_g=U_gP_{d(g)}$, for all
$g\in \mathcal{G}$, where
\begin{equation*}
  U_g : C^\infty(\mathcal{G}_{d(g)}) \rightarrow
  C^\infty(\mathcal{G}_{r(g)}), \,\ (U_gf)(g')=f(g'g).
\end{equation*}
Let $k_x$ be the distributional kernel of $P_x$, $x\in M$. Note that
the support of the $P$
\begin{equation*}
  \text{supp}(P):= \overline{\bigcup_{x\in M}\text{supp}(k_x)}  \subset \{(g,g'), \ d(g)=d(g')\} \subset \mathcal{G}\times \mathcal{G}
\end{equation*}
  since $\text{supp}(k_x)\subset
  \mathcal{G}_x\times\mathcal{G}_x$. Let $\mu_1(g',g) :=
  g'g^{-1}$. The family $P = (P_x)$ is called \emph{uniformly
    supported} if its \emph{reduced support} $\text{supp}_\mu(P) :=
  \mu_1(\text{supp}(P))$ is a compact subset of $\mathcal{G}$.

\begin{definition}\label{def.C*}
 The space $\Psi^{m}(\cG)$ of \emph{pseudodifferential operators of
   order $m$ on a Lie groupoid} $\mathcal{G}$ with units $M$ consists
 of smooth families of pseudodifferential operators $P=(P_x)$, $x\in
 M$, with $P_x\in \Psi^m(\mathcal{G}_x)$, which are {uniformly
   supported} and {right invariant}.
\end{definition}
We also denote $\Psi^\infty(\mathcal{G}) := \cup_{m\in
  \mathbb{R}}\Psi^m(\mathcal{G})$ and $\Psi^{-\infty}(\mathcal{G}) :=
\cap_{m\in \mathbb{R}}\Psi^m(\mathcal{G})$. We then have a
representation $\pi$ of $\Psi^{\infty}(\cG)$ on $\cC^\infty_c(M)$ (or
on $\cC^\infty(M)$, on $L^2(M)$, or on Sobolev spaces), called
\emph{vector representation} uniquely determined by the equation
\begin{equation}
  (\pi(P)f)\circ r := P(f\circ r),
\end{equation} 
where $f\in \cC^\infty_c(M)$ and $P=(P_x)\in \Psi^m(\cG)$.

An alternative definition of $\Psi^m(\cG)$ is through distribution
kernels (see for instance \cite{ALN, MonthubertPierrot, NWX}). More
precisely, $k_P(g):=k_{d(g)}(g, d(g))$ defines a distribution on
$\cG$, with $\text{supp} k_p= \text{supp}_\mu(P)$ compact, smooth
outside $M$ and given by an oscillatory integral on a neighborhood of
$M$. We say that $k_P\in I^m_c(\cG;M)$ is a conormal distribution to
$M$.  Conversely, we have $P_xf_x(g) = \int_{\cG_x} k_P(gh^{-1})
f_x(h)d\mu_x(h)$, and $\Psi^m(\cG)\cong I^m_c(\cG;M)$. If $P\in
\Psi^{-\infty}(\cG)$, then $P$ identifies with the convolution with a
smooth, compactly supported function and $\Psi^{-\infty}(\cG)$
identifies with the convolution algebra $\cC_c^\infty(\cG)$.  In
particular, we can define
\begin{equation}\label{eq.def.norm1}
  \|P\|_{L^1(\cG)} := \sup\limits_{x\in M} \Big\{ \
  \int_{\cG_x}|k_P(g^{-1})|\, d\mu_x(g),\,\, \int_{\cG_x}|k_P (g)|\,
  d\mu_x(g)\ \Big\}.
\end{equation}

There is an interesting representation, the \emph{ regular
  representation} $\pi_x$, associated to $x$ on $\cC^\infty_c(\cG_x)$,
defined by $\pi_x(P)=P_x$. It is clear that $\|\pi_x(P)\| \leq
\|P\|_{L^1}$. The \emph{reduced $C^*$--norm} of $P$ is defined by
\begin{equation}\label{eq.def.rednorm}
  \|P\|_r = \sup\limits_{x\in M}\|\pi_x(P)\| = \sup\limits_{x\in
    M}\|P_x\|,
\end{equation}
and the \emph{full norm} of $P$ is defined by
\begin{equation}\label{eq.def.fullnorm}
  \|P\| = \sup\limits_{\rho}\|\rho(P)\|,
\end{equation}
where $\rho$ varies over all bounded representations of $\Psi^0(\cG)$
satisfying
\begin{equation*}
  \|\rho(P)\| \leqslant \|P\|_{L^1}\quad \text{for all} \quad P\in
  \Psi^{-\infty}(\cG).
\end{equation*}

\begin{definition}\label{C*G}
Let $\cG$ be a Lie groupoid and $\Psi^\infty(\cG)$ be as above.  We
define $C^*(\mathcal{G})$ (respectively, $C^*_r(\cG)$) to be the
closure of $\Psi^{-\infty}(\mathcal{G})$ in the norm $\|\cdot\|$
(respectively, $\|\cdot\|_r$).  If $\|\cdot\|_r=\|\cdot\|$, that is,
if $C^*(\mathcal{G}) \cong C_r^*(\mathcal{G})$, we call $\cG$
\emph{amenable}.
\end{definition}

Now consider the closure $\overline{\Psi^0(\cG)}$ with respect to
$\|\cdot\|$. Let $A$ denote the algebroid defined by $\cG$, and $S^*A$
denote the sphere bundle of $A^*$. There is a well-defined principal
symbol mapping $\sigma: \overline{\Psi^0(\cG)}\to \cC_0(S^*A)$, which
is a surjective $*$-homomorphism and its kernel coincides with
$C^*(\mathcal{G})$:
\begin{equation}\label{ex.seq.symb}
\begin{CD}
  0 @>>> C^*(\cG) @>>> \overline{\Psi^0(\cG)}
  @>{ \sigma_0}>>
 \cC_0(S^*A)@>>> 0,
\end{CD}
\end{equation}
In particular $\Psi^{-\infty}(\cG)$ is dense in $\Psi^{-1}(\cG)$ and
$\Psi^{-1}(\cG)\subset C^*(\cG)$.  An operator $P\in
\overline{\Psi^0(\cG)}$ is said to be \emph{elliptic} if $\sigma_0(P)$
is invertible in $\cC_0(S^*A)$.

Let  $Y\subset M$ be an \emph{invariant} subset, that is, such
that $d^{-1}(Y)=r^{-1}(Y)$. Then, if $Y$ is a closed submanifold of
$M$, $\cG_Y:=d^{-1}(Y)$ is also a Lie groupoid, with units $Y$ and
there is an exact sequence
\begin{equation} \label{ex.seq.rest}
\xymatrix{ 
0 \ar[r]
  & C^*(\cG_{M\setminus Y})\ar[r]
      &C^*(\cG) \ar[r]
          &C^*(\cG_{Y})\ar[r] 
              &0.
}
\end{equation}
Moreover, there is a restriction map $\cR_Y: \Psi^{m}(\cG)\to
\Psi^{m}(\cG_Y)$.  In this case, Lemma 3 in \cite{LMN} gives that the
following sequence is exact:
\begin{equation}\label{ex.seq.inv}
\begin{CD}
  0 @>>> C^*(\cG_{M\setminus Y}) @>>> \overline{\Psi^0(\cG)}
  @>{(\cR_Y, \sigma)}>>
  \overline{\Psi^0(\cG_Y)}\times_{\cC_0(S^*A_{Y})}\cC_0(S^*A)@>>> 0,
\end{CD}
\end{equation}
where the fibered product $\overline{\Psi^0(\cG_Y)}
\times_{\cC_0(S^*A_{Y})} \cC_0(S^*A)$ is defined as the algebra of
pairs $(Q, f)\in \overline{\Psi^0(\cG_Y)} \times \cC_0(S^*A)$ such
that $\sigma(Q)=f_{\vert{S^*A_{Y}}}$. One possible strategy to prove
Fredholmness for pseudodifferential operators on $\cG$, in particular,
in $C^*(\cG)$, is to look for invariant subsets $Y\subset M$ such that
the $C^*$-algebra of $\cG_Y$ is (isomorphic to) the compact
operators. This is the case in the first example we consider
below. See \cite{LMN, NWX} for details.

\begin{example}[Pair groupoid]\label{expl_pair_grpd}
Let $M$ be a smooth manifold (with or without corners). Let
\begin{equation*}
\cG=M\times M \quad\quad \quad \mathcal{G}_0=M,
\end{equation*}
with structure maps $d(m_1, m_2)=m_2$, $r(m_1, m_2)= m_1$, $(m_1,
m_2)(m_2, m_3)=(m_1, m_3)$, $u(m)=(m,m)$, and $\iota(m_1, m_2)= (m_2,
m_1)$. Then $\cG$ is a Lie groupoid, called \emph{the pair
  groupoid}. We have $A(\cG)= TM$. According to the definition, a
pseudodifferential operator $P$ belongs to $ \Psi^m(\cG)$ if and only
if the family $P=(P_x)_{x\in M}$ is constant. Hence we obtain
$\Psi^m(\cG)= \Psi^m_{\text{comp}}(M)$. Also, an important result is
that $C^*(\mathcal{G}) \cong \mathcal{K}$, the ideal of compact
operators, the isomorphism being given by the vector representation or
by any of the regular representations (together with $\cG_x\cong M$.)
If $M$ has dimension $0$, say, it is a discrete set with $k$ elements,
then $C^*(\cG)\cong M_k(\C)$ and the convolution product becomes
matrix multiplication.
\end{example}

\begin{example}[Transformation (or Action) groupoid]\label{expl_transf_grpd}
Suppose that a Lie group $G$ acts on the smooth manifold $M$ from the
right. The \emph{transformation groupoid} over $M \times \{e\} \cong
M$, denoted by $M \rtimes G$, is the set $ M \times G $ with structure
maps $d(m,g )=(m \cdot g, e)$, $r(m, g)= (m,e)$, $(m, g)(m\cdot g,
h)=(m, gh)$, $u(m,e)=(m,e)$, and $\iota(m, g)= (m\cdot g, g^{-1})$.
For more on the action groupoid, one may see \cite{Mac, MM, Ren}.

Let $\mathfrak{g}$ be the Lie algebra of $G$. Denote by
$\mathfrak{X}(M)$ the space of smooth vector fields on $M$. The action
of $G$ on $M$ induces a Lie algebra homomorphism $\phi: \mathfrak{g}
\rightarrow \mathfrak{X}(M)$, i.e., an action of the Lie algebra
$\mathfrak{g}$ on $M$. The \emph{transformation Lie algebroid} $M
\times \mathfrak{g}$ has anchor map $\rho: M \times \mathfrak{g} \to
TM$ defined by
\begin{equation*}
\rho(m, v)= \phi(v)(m).
\end{equation*}
Any section $v$ of $M \times \mathfrak{g}$ is a map $v: M \rightarrow
\mathfrak{g}$. We define the bracket on sections of $M \times
\mathfrak{g}$ by
\begin{equation*}
  [v, w](m)= [v(m), w(m)]_{\mathfrak{g}}+ (\phi (v(m)) \cdot w)(m) -
  (\phi (w(m)) \cdot v)(m).
\end{equation*}
For more details, see \cite{CannasWeinstein}.  In general, there is no
obvious description of pseudodifferential operators on transformation
groupoid which depends on the action of $G$ on $M$. One case of
interest is $\cG= [0,\infty)\rtimes \R^+$, where $\R^+$ acts by
dilation. In this case, it is known that $C^*(\cG)$ coincides with the
class of Wiener-Hopf operators and we have $C^*(\cG)=
\cC_0([0,\infty))\rtimes \R^+$ \cite{MRen} (see also Example 5.11 and
the proof of Lemma 10.2 in \cite{LN}). Moreover, the anchor map $\rho:
[0,\infty)\times \R^+ \to T[0,\infty)$ is such that $\rho(0,
    \lambda)=0$, for all $\lambda \in \R^+$ and is injective
    otherwise, so we have
\begin{equation*}
\Gamma(A(\cH)) \cong  \{a(x)x\pa_x, a \in C^\infty([0,\infty))\},
\end{equation*}
the vector fields that vanish at $0$, that is, at the boundary.
\end{example}

The following example is of a different nature, in that our starting
point is the Lie algebroid.

\begin{example}[$b$-groupoid]\label{bgrpd}

Let $M$ be a manifold with smooth boundary and let $\cV_b$ denote the
class of vector fields on $M$ that are tangent to the boundary.  The
following construction is due to Melrose and led to the general
concept of $b$-geometry \cite{MelroseAPS}, and later to the more
general definition of Lie manifolds. The associated groupoid was
defined in \cite{Monthubert, NWX}.

According to the Serre-Swan theorem (\cite{Karoubi} see also
\cite{ALN, MelroseScattering, MelroseAPS}), there exists a smooth vector bundle
${^bTM} \rightarrow M$ together with a natural map of vector
bundles 
\begin{equation*}\begin{CD}
{^bTM} @> \rho  >> TM \\
@VV V 		@VV V \\
M	@= M
\end{CD}\end{equation*}
such that $\cV_b = \rho( \Gamma({^bTM})$. We call ${^bTM}$ the
\emph{b-tangent bundle}.  Since the Lie bracket of vector fields
tangent to a submanifold is again tangent to that submanifold, we see
that $\cV_b$ is a Lie algebra and ${^bTM}$ becomes a Lie algebroid.

Let 
$$ 
    \cG_b:= \bigcup\limits_{j} \R^+\times (\pa_j M)^2 
    \quad \bigcup \quad M_0^2,
$$ 
where $M_0^2$ denotes the pair groupoid of $M_0:= int(M)$ and $\pa_j
M$ denote the connected components of $\pa M$. Then $\cG_b$ can be
given the structure of a Lie groupoid with units $M$ and we have that
it integrates ${^bTM}$, that is, $A(\cG_b)= {^bTM}$. The
pseudodifferential calculus obtained is Melrose's $b$-calculus. (In fact,
Melrose's calculus is a little bit larger.) See
\cite{MelroseAPS, Monthubert, Monthubert1, NWX} for details. 
\end{example}


\smallskip
\vskip 3ex

\section{Domains with Conical Points, Desingularization, and Sobolev spaces}\label{weightedspaces}

We review here the main concepts needed regarding domains with conical
singularities.

\begin{definition}\label{domain}
Let $\Omega \subset \R^n$, $n\geqslant 2$, be an open connected
bounded domain. We say that $\Omega$ is a \emph{domain with conical
  points} if there exists a finite number of points $\{p_1, p_2,
\cdots, p_l\} \subset \partial \Omega$, such that
\begin{enumerate}
\item[(1)] $\partial \Omega \backslash \{p_1, p_2, \cdots, p_l\}$ is smooth;
\item[(2)] for each point $p_i$, there exist a neighborhood $V_{p_i}$
  of ${p_i}$, a possibly disconnected domain $\omega_{p_i} \subset
  S^{n-1}$, $\omega_{p_i} \neq S^{n-1}$, with smooth boundary, and a
  diffeomorphism $\phi_{p_i}: V_{p_i} \rightarrow B^{n}$ such that
$$\phi_{p_i}(\Omega \cap V_{p_i})=\{rx': 0<r< 1, x'\in \omega_{p_i}\}.$$
\end{enumerate}
If, moreover, $\pa\Omega= \pa\overline{\Omega}$, then we say that
$\Omega$ is a \emph{domain with no cracks}. The points $p_i$, $i= 1,
\cdots, l$ are called \emph{conical points} or \emph{vertices}. If
$n=2$, $\Omega$ is said to be a {\em polygonal domain}.
\end{definition}
(We can always assume that $\overline{V_i}\cap
\overline{V_j}=\emptyset$, for $i\neq j$, $i,j\in\{1,2,\cdots, l\}$,
and that $J\phi_k(0)=I_n,$ where $J\phi(0)$ is the Jacobian matrix of
$\phi_i$ at $p_i$.)

\begin{remark}
Let $V=\{p_1, p_2, \cdots, p_l\}$ be the given set of conical points
of a conical domain with no cracks $\Omega$, as above. The set $V$
does not determine the structure of $\Omega$ since we can always
increase it, but the minimum set of conical points is unique and
coincides with the singularities of $\pa \Omega$. These are \emph{true
  conical points} of $\Omega$. The other points in $V$ will be called
\emph{artificial points} and are the ones for which $\omega_{p_i}$ is
diffeomorphic to a hemisphere $S^{n-1}_+$. (See Remark 1 in
\cite{MN}.) Note that in fact for any $x\in \overline{\Omega}$, we can
take a neighborhood $V_x$ of $x$, a domain $\omega_x\subset S^{n-1}$
and a diffeomorphism $\phi_x$ such that $\phi_{p_i}(\Omega \cap
V_{p_i})=\{rx': 0<r< 1, x'\in \omega_{p_i}\}$. Then $x\in \Omega$ if,
and only if, $\omega_x=S^{n-1}$, and $x$ is a smooth boundary point
if, and only if, $\omega_x\cong S_+^{n-1}$. Otherwise, $x$ is a true
conical point. It is often useful in applications to boundary value
problems on $\Omega$ to regard smooth boundary points as (artificial)
vertices, representing for instance a change in boundary conditions.
\end{remark}

The condition $\pa\Omega= \pa\overline{\Omega}$ means that no boundary
point of $\Omega$ becomes an interior point of the closure, that is,
all boundary points of $\Omega$ are accessible from the outside
\cite{Fol}. We call $x\in \pa \Omega$ a \emph{crack point} if, using
the notation in the previous remark, $\pa \omega_x \neq \pa
(\overline{\omega_x} )$. In this setting, $x$ is a smooth boundary
point if, and only if, $\omega_x\cong S^{n-1}_+$ \emph{or}
$\omega_x\cong S^{n-1}_+ \sqcup S^{n-1}_+$, in which case $\Omega$
lies on both sides of $\pa \Omega$ close to $x$. We call $x$ a
\emph{smooth crack point}. The remaining crack points are singular
points of the boundary of the conical domain, that is, are true
vertices.  In Section \ref{crack}, we will allow polygonal domains to
have cracks (Definition \ref{def.crack}). Note that in higher dimensions, domains with cracks will have edges, and thus are no longer conical domains (in dimension two, the edges of the crack curves behave like conical points, so our constructions apply). See also \cite{LiMN,
  MN}.  \\

For the remainder of this section, we shall denote by $\Omega$ a
bounded domain with conical points in $\R^n$, with \emph{no
  cracks}. We now define the desingularization $\Sigma(\Omega)$ of
$\Omega$ and $\pa \Omega$, which will play a major role in our
constructions.  We follow the approach in \cite{BMNZ} (Section 4 and Example 2.11), see also \cite{Kon, MelroseAPS}.  Let $\Omega^{(0)}=\{p_1,p_2,\cdots, p_l\}$ be the
set of conical points of $\Omega$ and $\phi_i$, $\omega_i$ be as in
Definition \ref{domain}, for $i=1, \cdots, l$. Let also, for each $i$,
$\psi_i$ denote a smooth function on $\Omega$ such that $0\leq
\psi_i\leq1$, $\psi_i=1$ on $\phi^{-1}(\{rx': 0<r< \varepsilon_i,
x'\in \omega_{p_i}\}$, for some $\epsilon_i<1$, and $\psi_i=0$ outside
$V_i\cap \Omega$. Define a map $\Phi: \Omega \setminus \Omega^{(0)}
\rightarrow \R^{2n}$ by
\begin{equation*}
\Phi(x):= \left(x, \;\sum_i \psi_i(x) |x-p_i|^{-1} x\right).
\end{equation*}
Then $\Sigma(\Omega)$ is defined as the closure in $\R^{2n}$ of the
image of $\Phi$. The desingularization map $\kappa : \Sigma(\Omega)\to
\overline\Omega$ is given by projection on the first $n$
components. Note that $\kappa^{-1}(p_i)=\{p_i\}\times
\overline{\omega_i}$. Note also that if $x\in V_i\cap \Omega$ is
identified with $rx'$, $r\in (0, \varepsilon_i)$, $x'\in \omega_i$,
then $\Phi(x)=(rx', x')$ and we have that $\Phi(V_i'\cap \Omega)\cong
(0, \varepsilon_i)\times \omega_i$, for some $V_i'\subset V_i$, open.
We have then the following isomorphism:
\begin{equation}\label{desing}
   \Sigma(\Omega) \cong \left(\coprod\limits_{p_i \in
     \Omega^{(0)}}[0,\epsilon_{p_i})\times \overline{\omega_{p_i}}
     \right) \bigcup\limits_{\phi_{p_i}, \, p_i\in\Omega^{(0)}}
     \Omega,
\end{equation}
where the two sets are glued by $\phi_i$ along $V_i'\cap \Omega$. So
we see that $\Sigma(\Omega)$ is obtained from $\Omega$ by removing a
(possibly non-connected) neighborhood of the singular points and
replacing (each connected component) it by a cylinder.  The advantage
of using the approach above is that the results in \cite{BMNZ} become
available, in particular, the fact that $\Sigma(\Omega)$ is a Lie
manifold with boundary, that is, a compact Riemannian manifold with
corners with a given Lie algebra of vector fields tangent to the
boundary defining the metric, which enjoys many of the properties of
the smooth case \cite{ALNgeom, ALN}.
      
Note that it follows from (\ref{desing}) that the boundary is given by
\begin{equation}
  \pa \Sigma(\Omega) \cong \left(\coprod\limits_{p_i \in
    \Omega^{(0)}}[0,\epsilon_{p_i})\times \partial\omega_{p_i} \cup
    \{0\}\times \overline{\omega_{p_i} } \right)
    \bigcup\limits_{\varphi_{p_i}, \, p_i\in\Omega^{(0)}} \Omega_0.
\end{equation}
where $\Omega_0$ denotes the smooth part of $\partial\Omega$, that is,
$\Omega_0:=\partial \Omega \backslash \Omega^{(0)}$. There are
different types of hyperfaces of $\pa \Sigma(\Omega)$. Namely, each
face of $\Omega$ yields a hyperface, and to each $p_i$ we have
hyperfaces $\{0\}\times \overline{\eta_{ij} }$ and
$[0,\epsilon_{p_i})\times \chi_{ik}$, where $\eta_{ij}$ and
  $\chi_{ik}$
denote connected components of $\omega_i$ and $\pa \omega_i$,
respectively.  In the terminology of \cite{ALN}, the hyperface
$[0,\infty)\times\overline{\chi_{ik}}$ is {\em not at infinity} since it
  corresponds to an actual face of $\Omega$. The hyperface
  $\{0\}\times \overline{\eta_{ij}}$ is an hyperface {\em at infinity}
  because it corresponds to a singularity of $\Omega$.  We have a
  decomposition
$$ 
   \pa \Sigma(\Omega) = \partial' \Sigma (\Omega) + \partial''
     \Sigma (\Omega) 
$$ 
where $\partial' \Sigma (\Omega)$ denotes the union of hyperfaces
which are not at infinity and $\partial''
\Sigma(\Omega):=\kappa^{-1}(\Omega^{(0)})$ denotes the union of the
hyperfaces at infinity.  We remark that in fact $\partial' \Sigma
(\Omega)$ can be identified with the desingularization of $\pa \Omega$
and $\partial' \Sigma (\Omega)$ is a Lie manifold (without boundary)
\cite{BMNZ}. (Note that our definition of $\partial' \Sigma (\Omega)$
differs from the one in \cite{BMNZ}, in that here we are considering
the closure.)

The space $L^2(\Sigma(\Omega))$ is defined using the volume element of
a compatible metric with the Lie structure at infinity on
$\Sigma(\Omega)$. A compatible metric is $r_{\Omega}^{-2}g_e$, where
$g_e$ is the Euclidean metric and $r_{\Omega}$ is a weight function as
in \cite{BMNZ}, representing the distance to the singular points. Then
the Sobolev spaces $H^m(\Sigma(\Omega))$ are defined using
$L^2(\Sigma(\Omega))$. It happens that these Sobolev spaces are
related to \emph{weighted} Sobolev spaces on $\Omega$ and $\pa
\Omega$.

Let $m\in \mathbb{Z}_{\geqslant 0}$, $\alpha$ be a multiindex, and
$r_{\Omega}$ be a weight function as in \cite{BMNZ}. The $m$-th
Sobolev space on $\Omega$ with weight $r_{\Omega}$ and index $a$ is
defined by
\begin{equation}\label{def.weighted}
  \cK_{a}^m(\Omega)=\{u\in L^2_{\text{loc}}(\Omega), \, \,
  r_{\Omega}^{|\alpha|-a}\partial^\alpha u\in L^2(\Omega), \,\,\,\text{for
    all}\,\,\, |\alpha|\leq m\}.
\end{equation}

The following result can be found in \cite{BMNZ} (Proposition 5.7 and
Definition 5.8)

\begin{proposition} \label{Identification}
Let $\Omega \subset \R^n$ be a domain with conical points and
$\Sigma(\Omega)$ be its desingularization. Let also
$\partial'\Sigma(\Omega)$ be the union of the hyperfaces that are not
at infinity. Then,

\begin{enumerate}
\item[(a)] $ \cK^{m}_{\frac{n}{2}}(\Omega)\cong H^{m}(\Sigma(\Omega),
  g),$ for all $m\in \mathbb{Z}$.
\item[(b)] $ \cK^{m}_{\frac{n-1}{2}}(\partial\Omega)\cong
  H^{m}(\partial'\Sigma(\Omega))$, for all $m\in \mathbb{Z}_{\geqslant
    0}$.
\end{enumerate}
\end{proposition}

Note that the weighted Sobolev spaces on the boundary $\pa \Omega$ are
defined using the identification given in $(a)$ in the above
proposition. For more details, see \cite{BMNZ}.

{F}rom the previous proposition, it follows that results on regularity
and Fredholmness of operators on $\partial'\Sigma(\Omega)$ translate
into results for boundary operators.

\smallskip
\vskip 3ex

\section{Lie Algebroids and Lie Groupoids for Straight Cones\label{straightcone}}

Let $\omega \subset S^{n-1}$ be an open subset with smooth
boundary. We allow $\omega$ to be \emph{disconnected}. Denote by
$\Omega := \{t y',\ y' \in \omega,\ t \in (0, \infty)\} = \R^+ \omega$
the cone with base $\omega$. We use here the concepts reviewed in the
previous sections to place the approach in \cite{NQ1} in a setting
easier to generalize to domains with conical points.

We first desingularize our domain with conical points $\Omega$ and its
boundary $\pa \Omega$ as in the previous section. We obtain a Lie
manifold $\Sigma(\Omega)$, which in this case coincides with an
half-infinite solid cylinder
\begin{equation*}
\Sigma(\Omega)=  [0,\infty) \times \overline{\omega}
\end{equation*} 
with boundary $\pa \Sigma(\Omega) = [0,\infty) \times \partial \omega
\cup \{0\}\times \omega$.  We denote by $\partial'
\Sigma(\Omega)=[0,\infty)\times \partial \omega$ the union of the
hyperfaces not at infinity and by $\partial''
\Sigma(\Omega)=\{0\}\times \omega$ the union of the hyperfaces at
infinity (recall that $\omega$, $\pa \omega$ might be disconnected) .
Throughout this section, we let $M := \partial' \Sigma(\Omega) =
[0,\infty)\times \partial\omega$, which we think of as a
  desingularization of $\pa \Omega$.

The Lie algebra of vector fields $\mathcal{V}(\Omega)$ inducing the
Lie structure at infinity in $\Sigma(\Omega)$ is the space of vector
fields on $\Sigma(\Omega)$ that are tangent to $\{0\}\times
\overline{\omega}$, that is, to the boundary at infinity \cite{BMNZ,
  MN}.
There is an induced Lie structure at infinity on $M$, that can be
described as follows. If we denote by $p_1 : M \to [0,\infty)$ and
  $p_2 : M \to \pa \omega$ the two projections and decompose $ TM =
  p_1^* T[0, \infty) \oplus p_2^* T\pa \omega,$ we have that $\cW$
    consists of vector fields on $M = [0, \infty) \times \pa \omega$
      such that
\begin{equation*}
  \cW = \{a(x,y)x \partial_x + Y, \, 
  a \in C^\infty(M),\ Y \in \Gamma(p_2^* T\pa \omega)\, \},
\end{equation*}
that is, that are tangent to $\{0\} \times \pa \omega=\pa M$. Hence,
the Lie algebroid associated to $M$ as a Lie manifold is (isomorphic
to) the $b$-tangent bundle ${^bTM}$ such that $\Gamma({^bTM})=\cW$
(see Example \ref{bgrpd}).

Let us consider now the action of $(0, \infty)$ on $[0, \infty)$ by
  dilation and let
\begin{equation}\label{cH}
\cH := [0,\infty)\rtimes (0,\infty)
\end{equation}
be the transformation (or action) groupoid (Example
\ref{expl_transf_grpd}).  Also, we denote by $X^2 := X \times X$ the
pair groupoid with units some arbitrary space $X$ (Example
\ref{expl_pair_grpd}) and by $\cK$ the $C^*$-algebra of compact
operators on some generic separable Hilbert space. Then $C^*( X^2 )
\simeq \cK $, if $\dim X>0$, and $C^*( X^2 )\simeq M_k(\C)$, if $X$ is
discrete with $k$ elements.

We define the \emph{boundary groupoid associated to a straight cone
  $\Omega=\R^+\omega$} as the product Lie groupoid with units
$M=[0,\infty) \times \pa \omega$, corresponding to a desingularization
  of $\pa \Omega$, as
\begin{equation}\label{cJ} 
\cJ := \cH \times (\partial\omega)^2.
\end{equation}
Noting that $(0,\infty)$ is an invariant subset where $\cH$ coincides
with the pair groupoid, we have that $M_0:=int(M)= (0,\infty)\times
\pa \omega$ and $\pa M=\{0\} \times \pa \omega$ are invariant subsets
of $M$ with respect to $\cJ$ and we have
$$\cJ_{M_0}=\left((0,\infty) \times \pa \omega\right)^2=M_0^2, \quad
\cJ_{\pa M}= (0,\infty) \times (\partial\omega)^2.$$ (Note that $M_0$
is really the smooth part of $\pa \Omega$.)

We see now that that $\cJ$ integrates ${^bTM}$, that is, that the
sections of $A(\cJ)$ are vector fields tangent to $\pa M$:

\begin{lemma}\label{lem.algbd.cJ}
The \text{Lie algebroid} of $\cJ $ is isomorphic to ${^bTM}$, the
$b$-tangent bundle.
\end{lemma}

\begin{proof}
We use the general fact that the Lie algebroid of the product groupoid
$\mathcal{G}_1\times \mathcal{G}_2$ is the product of their Lie
algebroids, i.e. $A(\mathcal{G}_1\times \mathcal{G}_2) =
A(\mathcal{G}_1)\times A(\mathcal{G}_2)$ (see Example 2.5 in
\cite{LN}).  We have that (see Example \ref{expl_transf_grpd})
\begin{equation*}
\Gamma(A(\cH)) \cong \cV:= \{a(x)x\pa_x, a \in C^\infty([0,\infty))\},
\end{equation*}
that is, the Lie algebroid of $\cH$ is isomorphic to the one coming
from $\cV$ through the Serre-Swan theorem. Since
$A((\pa\omega)^2)=T(\pa \omega)$, we get
\begin{equation*}
\Gamma(A(\cJ))=\Gamma(A(\cH) \times T(\pa\omega)) \cong \cW=\Gamma({^bTM})
\end{equation*}
and  the conclusion follows from the Serre-Swan theorem.
\end{proof}

 \begin{remark}\label{rmk.bgrpd}
Recall the definition of $b$-groupoid in Example \ref{bgrpd}, which,
in the case of $M=[0,\infty) \times \pa \omega$ comes down to
$$
  {^b\cG}= M_0^2 \; \bigcup_j \;\R^+\times (\pa_j \omega)^2,
$$ 
where $\pa_j \omega$ denote the connected components of $\pa \omega$.
If $\pa \omega$ is connected, then $\cJ={^b\cG}$. In the more
interesting case of $\pa \omega$ being disconnected, the groupoid
$\cJ$ is larger and not $d$-connected. (Note that if $n=2$, that is,
if we have a polygonal domain, then $\pa \omega$ is \emph{always}
disconnected.) The main difference is that here we allow the different
connected components of the boundary, corresponding to faces in the
desingularization, to interact, in that there are arrows between
them. In fact, we have that $\R^+\times (\pa_j \omega)^2$ is a
connected component of $\cJ_{\pa M}=\R^+\times (\pa \omega)^2$, and we
have $\R^+\times (\pa_j \omega)^2=\cJ_{\pa_j \omega}^{\pa_j \omega}$,
the subgroupoid over $\pa_j\omega$ defined as $d^{-1}({\pa_j
  \omega})\cap r^{-1}{\pa_j \omega}$. Hence ${^b\cG}$ is an open, wide
subgroupoid of $\cJ$.
\end{remark}

As for $C^*$-algebras, it is known that $C^*(\cH)=\cC_0([0,\infty))\rtimes \R^+$
  \cite{MRen}. We have then the following:

\begin{lemma}\label{lem.C*.cJ} 
Let $\cJ = \cH \times (\pa \omega)^2$. Then $C^*(\cJ)\cong
\cC_0([0,\infty))\rtimes \R^+\otimes \cK$, if $n\geq 3$, and $C^*(\cJ)\cong \cC_0([0,\infty))\rtimes \R^+\otimes
M_k(\C)=M_k(C^*(\cH))$, if $n=2$, where $k$ is the number of elements
of $\pa \omega$.
\end{lemma}

\begin{proof}
The $C^*$-algebra of the pair groupoid $(\partial\omega)^2$ is isomorphic to $\cK$, if $n\geq
3$, and to $M_k(\C)$, if $n=2$, where $k$ is the number of elements of
$\pa \omega$. Since (see Proposition 4.5 in \cite{LN})
\begin{equation*}
  C^*(\mathcal{J})\cong C^*(\cH)\otimes
  C^*(\partial\omega\times \partial\omega),
\end{equation*}
the result follows.
\end{proof}
 
The $C^*$-algebra associated to $\cH$ is the algebra of Wiener-Hopf
operators on $\R^+$, and its unitalization is the algebra of Toeplitz
operators  \cite{MRen}. In particular, we see that the groupoid $C^*$-algebra of a straight cone is stably isomorphic to the $C^*$-algebra of Wiener-Hopf operators. Let us recall the Toeplitz exact sequence
\begin{equation*} \xymatrix{ 
0 \ar[r]
  & \cK(L^2(\mathbb{R}^+))\ar[r]
      &C^*(\cH) \ar[r]
          &\cC_0((0,\infty)) \ar[r] 
              &0.
}
\end{equation*}
Tensoring with $\cK$, we obtain another short exact sequence, for
$n\geq 3$,
\begin{equation*} \xymatrix{ 
0 \ar[r]
  & \cK\ar[r]
      &C^*(\cJ) \ar[r]
          &\cC_0((0,\infty)) \otimes \cK\ar[r] 
              &0
}
\end{equation*}
and similarly for $n=2$. In fact, the sequence above can be obtained
from (\ref{ex.seq.rest}), for the invariant, closed subset $\pa M$
(see Proposition \ref{Groupoid} below).

As for pseudodifferential operators, note that if $P\in \Psi^m(\cJ)$,
then $P_{M_0}\in \Psi^m(M_0^2)$, so that it identifies, by invariance,
with an operator in $ \Psi^m(M_0)=\Psi^m(\R^+\times \pa \omega)$. At
the boundary, $P_{\pa M}\in \Psi^m(\R^+\times (\pa \omega)^2)$ is
defined by a distribution kernel $k_P$ in $\R^+\times (\pa
\omega)^2)$. If $P_{\pa M}\in \Psi^{-\infty}(\R^+\times (\pa
\omega)^2)$, that is, if $k_p$ is smooth, then it defines a smoothing
\emph{Mellin convolution operator} on $\R^+\times \pa \omega$ (see
\cite{LP,NQ1}). This is in fact one of the reasons behind our
definition of $\cJ$.

Moreover, it follows from Remark \ref{rmk.bgrpd} that
$\Psi^m(\cJ)\supset \Psi({^b\cG})$, Melrose's $b$-calculus. If we let
$P=(P_x)\in \Psi^m(\cJ)$, and, for $x\in \pa M$, write $P_x\in
\Psi^m(\R^+\times \pa \omega)$ as matrix of operators $P_x^{ij}:
\cC_c^\infty(\R^+\times \pa_j \omega) \to \cC_c^\infty(\R^+\times
\pa_i\omega) $, then the $b$-pseudodifferential operators at the
boundary correspond to the diagonal entries.  We call
$\Psi^\infty(\cJ_{\pa M})$ the \emph{indicial algebra}, in analogy
with the usual indicial algebra for the $b$-calculus (which coincides
with $\Psi^\infty({^b\cG}_{\pa M})$)\cite{MelroseAPS, MelroseN,
  Monthubert}.

In the following section, we generalize these constructions to the setting of domains with conical points.


\smallskip
\vskip 3ex

\section{Boundary Groupoids for Domains with Conical Points\label{GpdConstruction}}
\vskip 2.5ex
\subsection{Groupoid construction for domains with no cracks}\label{highdimension}

We consider now a bounded domain $\Omega$ with conical points as in Definition \ref{domain}. Our goal gere is to define a groupoid over the desingularization of $\pa \Omega$, generalizing Equation (\ref{cJ}). For now, $\Omega$ is not allowed to have cracks, but we will drop that assumption in the following subsection.

Let $\Omega^{(0)}=\{p_1,p_2,\cdots, p_l\}$ be the set of conical points of
$\Omega$ (for simplicity, we assume that all the $p_i$'s are true conical points). Denote by $\Omega_0$ the smooth part of $\partial\Omega$, that is,
$\Omega_0=\partial \Omega \backslash \Omega^{(0)}$. Suppose that $p_i\in \Omega^{(0)}$ is
a conical point. Then there exist a neighborhood $V_i \subset
\mathbb{R}^n$ of $p_i$ and a diffeomorphism $\phi_i: V_i \rightarrow
B(0,1)\subset \mathbb{R}^n$, such that, up to a local change of
coordinates, $
  J\phi_k(0)=I_n,
$
where $J\phi(0)$ is the Jacobian matrix of $\phi_i$ at $p_i$, and
\begin{equation*}
  \phi_i (V_i\cap \Omega) = B(0,1)\cap \mathbb{R}_+ \omega_i,
\end{equation*}
for some smooth domain $\omega_i\subset S^{n-1}$. Furthermore, we can
assume that $\overline{V_i}\cap \overline{V_j}=\emptyset$, for $i\neq j$, $i,j\in\{1,2,\cdots, l\}$.
As before, we denote by  $\Sigma(\Omega)$ the manifold with corners obtained from the desingularization of $\Omega$, as in Section \ref{weightedspaces} and by $M:=\pa'\Sigma(\Omega)$ the union of the hyperfaces not at infinity, which we can identify with the desingularization of $\pa \Omega$.

We now construct a groupoid from $\Omega$ and $\partial\Omega$.  {F}rom
the discussion in Section \ref{straightcone}, for each conical point
$p_i \in \partial\Omega$, we can construct groupoid $\cJ_{p_i}=\cH
\times (\partial\omega_{p_i})^2$, where $\cH = [0,\infty)\rtimes
  (0,\infty)$. Since the transformation groupoid $(0,\infty)\rtimes
  (0,\infty)$ is isomorphic to the pair groupoid $(0,\infty)\times
  (0,\infty)$ and $[0,\infty)\times (0,\infty)$ is diffeomorphic to
    $[0,\epsilon_{p_i})\times(0,\epsilon_{p_i})$, for any
      $0<\epsilon_{p_i}<1$, we can impose a groupoid structure on
      $\cJ'_{p_i}:=([0,\epsilon_{p_i})\times(0,\epsilon_{p_i}))\times
        (\partial\omega_{p_i}\times\partial\omega_{p_i})$ such that
\begin{enumerate}
\item[(a)] the unit space $M_{p_i}'$ of $\cJ'_{p_i}$ is
  $[0,\epsilon_{p_i})\times \partial\omega_{p_i}$,
\item[(b)] $[0,\epsilon_{p_i}) \times (0,\epsilon_{p_i})$ has the same
  transformation groupoid structure as that of
  $[0,\infty)\rtimes(0,\infty)$;
\item[(c)] the interior of $\cJ'_{p_i}$ is isomorphic to the pair
  groupoid of $(0,\epsilon_{p_i})\times\partial\omega_{p_i}$.
\end{enumerate}

Let $\Omega_0$ be the smooth part of $\partial\Omega$ as above and
$\Omega_0^2$ be the pair groupoid. For each $p_i \in \Omega^{(0)}$,
there is a map $(0,\epsilon_{p_i})\times\partial\omega_{p_i} \to
\Omega_0$, $(t,x)\mapsto \phi_i^{-1}(xt)$, and it follows from (c)
that we can define $\varphi_{p_i}: int \left(\cJ'_{p_i}\right)
\rightarrow \Omega_0^2$ by
\begin{equation*}
  \varphi_{p_i}(t,s, \omega_1,\omega_2) = (\phi_i^{-1}(t\omega_1), \phi_i^{-1}(ts\omega_2)),
\end{equation*}
where $t,s\in (0,\epsilon_i)$ and $\omega_1,\omega_2 \in
\partial\omega_{p_i}$.  It is clear that $\varphi_{p_i}$ is smooth, a
diffeomorphism into its image, and preserves the groupoid structure of
$\cJ'_p$ and $\Omega_0^2$. We can glue $\cJ'_{p_i}$ and $\Omega_0^2$
using the function $\varphi_{p_i}$, and we define the \emph{boundary
  groupoid associated to a conical domain $\Omega$} (with no cracks)
as
\begin{equation}\label{grpd.nocrack}
  \cG:=\left(\coprod\limits_{p_i \in \Omega^{(0)}}\cJ'_{p_i} \right)\quad
  \bigcup\limits_{\varphi} \quad
 \Omega_0^2,
\end{equation}
where $\varphi=(\varphi_{p_i})_{p_I\in \Omega^{(0)}}$. Then $\cG$ is a
Lie groupoid, with space of units
\begin{eqnarray}\label{units.nocrack}
  M &=& \left(\coprod\limits_{p_i \in
    \Omega^{(0)}}[0,\epsilon_{p_i})\times \partial\omega_{p_i} \right)\quad
    \bigcup\limits_{\varphi} \quad \Omega_0 \quad  \cong \quad\partial' \Sigma (\Omega),
\end{eqnarray}
where $\partial' \Sigma (\Omega)$ denotes the union of hyperfaces
which are not at infinity of a desingularization, as in Section
\ref{weightedspaces}. Clearly, the space $M$ of units is
compact. Denoting by $M_0$ the interior of $M$, we have
$M_0=\Omega_0$, so $\Omega_0$ is an open dense subset of $M$.

(We often replace $\cJ'_{p_i}$ by $\cJ_{p_i}$ in the definition of
$\cG$, where the gluing is always meant as above.)

\begin{definition}\label{def.layerC*}
The \emph{layer potentials $C^*$-algebra} associated to a conical domain $\Omega$ is defined as the groupoid $C^*$-algebra $C^*(\cG)$, where $\cG$ is the boundary groupoid as in (\ref{grpd.nocrack}).
\end{definition}

In the following results, we give a few properties of the boundary groupoid
$\cG$ and of the layer potentials $C^*$-algebra.

\begin{proposition}\label{Groupoid}
Let $\cG$ be the groupoid (\ref{grpd.nocrack}) associated to a domain
with conical points $\Omega\subset \R^n$. Let
$\Omega^{(0)}=\{p_1,p_2,\cdots, p_l\}$ be the set of conical points
and $\Omega_0=\partial \Omega \backslash \Omega^{(0)}$ be the smooth
part of $\partial\Omega$.  Then, $\cG$ is a Lie groupoid with units
$M= \partial' \Sigma (\Omega)$ such that
\begin{enumerate}

\item   $M_0=int(M) = \Omega_0$ is an invariant subset and
$$\cG_{M_0} \cong M_0 \times M_0.$$

\item For each conical point $p \in \Omega^{(0)}$, $\{ p \} \times \pa
\omega_p$ and $\pa  M= \bigcup\limits_{p \in \Omega^{(0)}} \{ p \} \times \pa
\omega_p$ are invariant subsets and
\begin{equation*}
\cG|_{\pa M}= \coprod\limits_{i=1}^{l} (\pa\omega_i \times \pa\omega_i) \times (\R^+ \times \{p_i\})
\end{equation*} 

\item $A(\cG)\cong {^bTM}$, the $b$-tangent bundle of $M$.
\end{enumerate}
\end{proposition}

\begin{proof}
To show (1) and (2) note that if $m \in M_0$, then $\cG_m = \partial
\Omega \backslash \Omega^{(0)} = \Omega_0$ and that if $x \in \partial
M$, then $\cG_x \cong (0,\infty) \times \pa \omega_i$ for some $i \in
\{1,\cdots, l\}$.

To compute the Lie algebroid of $\cG$, we see that if $\cG=
\cG_1\cup_\phi \cG_2$, for Lie groupoids $\cG_i$, $i=1,2$, and $\phi$
a diffeomorphism of an open set in $\cG_1$ to an open set in $\cG_2$
preserving the groupoid structure, then we can do a clutching
construction on vector bundles to get that $A(\cG)\cong A(\cG_1)
\cup_\phi A(\cG_2)$. In our case, since $A( \Omega_0^2)=T(\Omega_0)$
and $A({\cJ}'_i)= {^b T}([0,\infty)\times \pa \omega_i)$, according to
  Lemma \ref{lem.algbd.cJ}, we have that
$$ A(\cG)= \left(\bigoplus \limits_{i=1}^{l} {^b T([0,\infty)\times \pa \omega_i)}\right)\quad
  \bigcup\limits_{\varphi} \quad
T\Omega_0,$$
and therefore $\Gamma(A(\cG)) $ coincides with the vector fields on $M$ that are tangent to $\{0\}\times \pa \omega_i$, $i=1, \cdots, l$, that is, tangent to $\pa M$. Hence, $A(\cG)={^bTM}$.
\end{proof}

Recall that $C^*(\cG)$ was defined as the closure (in the full norm)
of the class of order $-\infty$ pseudodifferential operators on $\cG$
(Definition \ref{C*G}).

\begin{proposition}\label{C*alg}
Let $\cG$ be the boundary groupoid (\ref{grpd.nocrack}) associated to
a domain with conical points $\Omega\subset \R^n$ as before. Then,

\begin{enumerate}
\item $\cG$ is amenable, i.e., $C^*(\cG)\cong C^*_r(\cG)$.

\item We have the following exact sequences
\begin{equation*} \xymatrix{ 
0 \ar[r]
  & \cK\ar[r]
      &C^*(\cG) \ar[r]
          &\bigoplus\limits_{i=1}^{l} \cC_0(\R^+) \otimes \cK\ar[r] 
              &0,
}\text{  if  } n\geq 3,
\end{equation*}
 and, if $n=2$, 
\begin{equation*}
\xymatrix{ 
0 \ar[r]
  & \mathcal{K}\ar[r]
      &C^*(\cG) \ar[r]
          &\bigoplus\limits_{i=1}^{l} M_{k_i}(\cC_0(\mathbb{R}^+)) \ar[r] 
              &0,
              }
\end{equation*}       
where $k_i$ is the number of elements of $\pa \omega_i$ and
  $l$ is the number of conical points.
\end{enumerate}
\end{proposition}

\begin{proof}
To prove amenability, we first note that the pair groupoid is amenable
(since the regular representations are essentially the only
representations). Moreover, the disjoint union of amenable groupoids
is amenable and the same holds for crossed products by amenable
groups. Since $\R^+$ is abelian, hence amenable, we have that $
[0,\infty)\rtimes (0,\infty) \times (\partial\omega_{p_i})^2$ is
  amenable. Since $\cG$ is given by the gluing of two amenable
  groupoids, it is amenable. (In fact, we only needed $\cG_{M_0}$,
  $\cG_{\pa M}$ amenable, see \cite{Ren}).

As for the exact sequences, since $\pa M$ is a closed, invariant
submanifold of $M$, we have from (\ref{ex.seq.rest}) that
\begin{equation*} \xymatrix{ 
0 \ar[r]
  &C^*(\cG|_{M\backslash \pa M}) \ar[r]
      &C^*(\cG) \ar[r]
          &C^*(\cG|_{\pa M}) \ar[r] 
              &0.
}
\end{equation*}
Assume first that $n\geq 3$. The $C^*$-algebra of the pair groupoid is
isomorphic, through the vector representation, to the ideal of compact
operators, and one can easily check that $C^*(\R^+ \times \{p_i\})=
\cC_0(\R)$. Hence, we obtain from (1) and (2) in Proposition
\ref{Groupoid},
$$ C^*(\cG_{M_0})=\cK, \quad \text{and} \quad C^*(\cG|_{\pa M}) =  \bigoplus\limits_{i=1}^{l} \cK \otimes \cC_0(\R^+)
$$ and the result follows in this case. If $n=2$, then $ \pa \omega_i$
are discrete sets, hence $C^*((\pa \omega_i)^2)\cong M_{k_i}(\C)$,
where $k_i$ is the number of elements of $\pa\omega_i$. The exact
sequence then follows in the same way.
\end{proof}

Taking into account the remarks at the end of Section
\ref{straightcone}, we can characterize the class $\Psi^m(\cG)$ of
pseudodifferential operators on our boundary groupoid:

\begin{proposition}
Let $\cG$ be the boundary groupoid (\ref{grpd.nocrack}) with units $M$
associated to a domain $\Omega\subset \R^n$ with conical points
$\Omega^{(0)}=\{p_1, \cdots, p_l\}$, $\Omega_0=\Omega \setminus
\Omega^{(0)}$ be the smooth part of $\pa\Omega$.  Then
\begin{enumerate}
  \item $\Psi^m(\cG_{M_0})\cong \Psi^m(\Omega_0)$.
  \item If $P\in \Psi^m(\cG_{\pa M})$ then for each $p_i\in \Omega^{(0)}$, $P$ defines a Mellin convolution operator on $\R^+\times \pa \omega_i$.
  \item $\Psi(\cG)\supset \Psi({^b\cG})$, the $b$-pseudodifferential operators on $M$. 
\end{enumerate}
\end{proposition}

\begin{proof}
{F}rom Proposition \ref{Groupoid}, we have $\cG_{M_0}\cong
(\Omega_0)^2$, so $(1)$ follows from invariance. Moreover, for $P\in
\Psi(\cG_{\pa M})$, we have that on each invariant subset
$\{p_i\}\times \pa \omega_i$, $P$ defines a kernel $\kappa_{i}\in
(\pa\omega_i )^2 \times (\R^+ \times \{p_i\})$, hence a Mellin
convolution operator on $\R^+\times \pa \omega_i$ with kernel
$\tilde{\kappa_{i}}(r,s, x',y'):=\kappa_i(r/s, x', y')$.

As for $(3)$, note that for $M$ as in (\ref{units.nocrack}), the
$b$-groupoid becomes
$${^b\cG}= M_0^2 \; \bigcup_{i,j} \;\R^+\times (\pa_j \omega_i)^2,$$
where $\pa_j \omega_i$ denote the (open) connected components of $\pa
\omega_i$, $i=1, \cdots, l$. Since
$$\R^+\times (\pa_j \omega_i)^2= d^{-1}(\pa_j \omega_i)\cap
r^{-1}(\pa_j \omega_i) = \cG^{\pa_j \omega_i}_{\pa_j \omega_i},$$ we
have that ${^b\cG}$ is an open subgroupoid of $\cG$. The result
follows.
\end{proof}

We remark that it follows from the proof that if $\pa \omega_i$ is
connected, for all $i=1, \cdots, l$, then $\cG={^b\cG}$ (see also
Remark \ref{rmk.bgrpd}).

In general, we also have that the $b$-indicial algebra
$\Psi^0({^b\cG}_{\pa M})$ is a *-subalgebra of the indicial algebra
$\Psi(\cG_{\pa M})$. If $Q=(Q_x)\in \Psi^0({^b\cG}_{\pa M})$, then for
$x\in \pa_j \omega_i\subset \pa M$, we have $Q_x\in \Psi^0(\R^+\times
\pa_j \omega_i)$, so that it defines $P_x\in \Psi^0(\R^+\times \pa
\omega_i)$ by $P_x^{kk'}=Q_x$, if $k=k'=j$, and $P_x^{kk'}=0$,
otherwise, where $P_x^{kk'}: C^\infty_c(\R^+\times \pa_{k'}
\omega_i)\to C^\infty_c(\R^+\times \pa_{k} \omega_i)$. For each $x\in
\pa M$, we then identify $Q_x$ with a diagonal entry in $P_x$,
regarded as a matrix of operators. (Note that if $n=2$, then we get a
matrix of operators on $\R^+$.)  The indicial algebra plays a key role
in Fredholmness of operators on $\cG$, as we shall see in Section \ref{Fredholmness}.

To finish this subsection, we now give a result that shows that the boundary groupoid and the layer potentials $C^*$-algebra only depend, up to equivalence, on the number of singularities of the conical domain $\Omega$. (Compare with a
similar result for the $b$-groupoid of a manifold with corners
\cite{Monthubert}). Recall that two groupoids $\cG$ and $\cG'$ are
\emph{equivalent} if there exists a groupoid $\cZ$ containing $\cG$
and $\cG'$ as full subgroupoids, that is, if $\cG =
\cZ_A^A:=d_\cZ^{-1}(A)\cap r_\cZ^{-1}(A)$ for some subset $A\subset
\cG^0$ intersecting all orbits, and the same for $\cG'$. (We also get
$\cG$ and $\cG'$ equivalent to $\cZ$.) In this case, it is an
important result of Muhly, Renault and Williams that the
$C^*$-algebras $C^*(\cG)$ and $C^*(\cG')$ are Morita equivalent (see
\cite{MRW} for the definitions and result).

\begin{theorem}\label{thm.equiv}
Let $\Omega, \, \Omega'\subset \R^n$ be two conical domains with no
cracks, and $\cG$, $\cG'$ be the boundary groupoids as in
(\ref{grpd.nocrack}). If $\Omega$ and $\Omega'$ have the same number
of true conical points then $\cG$ is equivalent to $\cG'$. Hence,
$C^*(\cG )$ is Morita equivalent to $C^*(\cG')$.
\end{theorem}

\begin{proof}
We construct a groupoid $\cZ$ that yields the equivalence. Let $p_i$,
$p'_i$, $i=1, \cdots, l$ be the (true) conical points of $\Omega$ and
$\Omega'$, respectively. Define
$$ Z:= (\Omega \sqcup \Omega')/\sim$$ where we identify $p_i\sim
p'_i$, for each $i=1, \cdots, l$. Then $Z$ is also a conical domain
with no cracks : we have $\pa Z \setminus \{[p_1], \cdots, [p_l]\} =
\Omega_0 \sqcup \Omega'_0$, with $\Omega_0$ and $\Omega'_0$ the smooth
parts of $\pa \Omega$ and $\pa \Omega'$, respectively, hence it is
smooth. If $x= [p_i]$, then $\omega_{p_i} \sqcup \omega_{p_i'}$ is a
basis for a local cone at $x$ (note that $\omega_{p_i} \sqcup
\omega_{p_i'}\neq S^{n-1}$ since $p_i$ and $p'_i$ are true vertices,
and by definition $\overline{\omega_{p_i} \sqcup \omega_{p_i'}}=
\overline{\omega_{p_i} }\sqcup \overline{\omega_{p_i'}}$, so $Z$ has
no cracks).

Let $\cZ$ be the boundary groupoid (\ref{grpd.nocrack}) associated to
$Z$. It follows from (\ref{units.nocrack}) that the units of $\cZ$
coincide with $M\sqcup M'$, where $M, M'$ are units of $\cG$,
$\cG'$. Moreover, we have $\cG= \cZ_M^M$ and we check that $M$
intersects all orbits. Over the interior of $M\sqcup M'$, the orbits
are given by $\Omega_0\cup \Omega'_0$, which intersects $M$. Over the
boundary, the orbit at $([p_i],x)$, with $x\in \pa \omega_{p_i} \cup
\pa \omega_{p'_i}$ is $(\omega_{p_i} \sqcup \omega_{p_i'})$, which
also intersects $M$.

Hence, $\cG$ is a full subgroupoid of $\cZ$.  Since the same holds for
$\cG'$ with respect to $M'$, we conclude that $\cG$ and $\cG'$ are
equivalent. The last assertion follows from the main result in
\cite{MRW}.
\end{proof}

\smallskip
\subsection{Groupoid construction for polygonal domains with ramified cracks}\label{crack}

In this subsection, we consider  a polygonal domain $\Omega\subset \R^2$. We allow $\Omega$ to have ramified cracks (see the definition below). To any such domain, we associate a generalized conical domain, the so-called \emph{unfolded domain} $\Omega^u$, with no cracks, and apply the groupoid construction of the previous section. 

We first give a precise definition of cracks. We distinguish between different kinds of crack points, even though we will basically need the distinction between  smooth and singular crack points. The case of conical crack points is also considered, where we need to distinguish between  the crack part and  the non-crack parts of the conical neighborhood. 

\begin{definition}\label{def.crack}
Let $\Omega\subset \R^2$ be a polygonal domain as in Definition \ref{domain}. Let $x \in \pa\Omega$. Suppose that there  exists a neighborhood $V_x$ of $x$ in $\Omega$, an open subset $\omega_x \subset S^1$, $\omega_x \neq S^1$, and a diffeomorphism $\phi_x: V_x \rightarrow B^2$ such that in polar coordinates $(r,\theta)$, we have 
$$\phi_x(V_x \cap \Omega)= \{(r,\theta), r\in (0,1), \theta \in \omega_x\}.$$
Then $x$ is a \emph{crack point} if $\pa \omega_x \neq \pa \overline{\omega_x}$. Moreover:
\begin{enumerate}
\item  $x$ is an \emph{inner crack point} if $\overline{\omega_x}= S^1$. If $\omega_x$ is isomorphic to $S^1_+\sqcup S^1_+$, where $S^1_+$ is the hemisphere, we call $x$ a \emph{smooth crack point}. 

\item $x$ is an \emph{outer crack point} if $\overline{\omega_x} = S^1_+$.
\item otherwise, $x$ is a \emph{conical crack point}. 
\end{enumerate}
\end{definition}

\begin{remark}
In the above definition, we shall always assume that crack curves intersect transversally. \end{remark}

\begin{remark}
By Definition \ref{def.crack}, any non-smooth point on the boundary $\pa\Omega$ can be classified to be one of true conical non-crack points, inner, outer, and conical crack points.
Moreover, it is easy to check that the set of smooth crack points is given by a union of smooth curves, the actual cracks, with boundary given by inner,  outer or possibly conical crack points. In higher dimensions, even if Definition \ref{def.crack} would still apply, a crack would be a codimension $1$ submanifold and its boundary would be a codimension $2$ submanifold, corresponding to non-smooth crack points of $\pa \Omega$. Hence $\Omega$ would not be a conical domain, in the sense that the set of singular points is not discrete (and finite). In this case, the boundary points of the crack behave like edge points, rather than conical points, and $\Omega$ would be the interior of a manifold with corners.
\end{remark}

To each crack point $x$, we now associate  a \emph{ramification number}, which will turn out to be very useful to define the unfolded domain as it gives us the number of different ways we can approach the boundary close to $x$. It is defined as follows.
\begin{enumerate}
\item If $x$ is a inner or outer crack point, we let $k$ be the number of connected components of $\omega_x$ (so $\omega_x$ is isomorphic to the disjoint union of $k$ copies of $(0,2\pi/k)$ or $(0,\pi/k)$, for inner and outer crack points, respectively). Then we say that $x$ is a crack point with $k$ \emph{ramifications}. (In particular, if $x$ is a smooth crack point, then its ramification number is $2$).
\item  If $x$ is a conical crack point, then we  split $\omega_x= \omega_x' \cup \omega_x''$, where $\omega_x'$ and $\omega''$ are disjoint and given by unions of connected components of $\omega_x$ such that $I_j \subset \omega'$ if and only if $\overline{I_j}\cap \overline{I_i} =\emptyset$ for all $i \neq j$, where $I_j$ are the connected components of $\omega$, so $\pa \omega' =  \pa \overline{\omega_x'}$ and $\pa \omega_x''\neq  \pa \overline{\omega_x''}$.  Note that $\omega''_x$ is non-empty (otherwise, we get a true conical non-crack point.) Let  $k$ be the number of connected components of $\omega_x''$. If $\omega'_x=\emptyset$, we say that  $x$ is a conical crack point with $k$ \emph{ramifications}. If $\omega'_x\neq \emptyset$, we say that  $x$ is a conical crack point with $k+1$ \emph{ramifications}. 
\end{enumerate}

 If $x$ is a conical crack point, we always refer to $\omega'_x$ and $\omega''_x$ as the no-crack and the crack part of $\omega_x$, respectively. For instance, if $\omega_x= (0,\pi/4) \cup (\pi/4, \pi/2)\cup (\pi, 5\pi/4)\cup (3\pi/2, 7\pi/4))$, then $\omega'_x= (\pi, 5\pi/4)\cup (3\pi/2, 7\pi/4))$, $\omega''_x= (0,\pi/4) \cup (\pi/4, \pi/2)$ and  its ramification number is $3$. Also, a point with ramification $1$ means that it is an end point of some crack curve.
We refer to conical/outer/inner crack points with $k\neq 2$, as \emph{singular crack points}, as they are true conical points in the sense of Definition \ref{domain} (that is, singularities of $\pa \Omega$).

For the remainder of this subsection,  we denote by $\Omega^{(0)}=\{p_1, \cdots, p_l\}$ the set of true conical non-crack points. (Note that here $\Omega^{(0)}$ does not include all conical points in the sense of Definition of \ref{domain}.) Also, we let $\cC:=\{c_1, \cdots, c_m\}$ be the set of singular crack points, with $c_1, \cdots, c_{m'}$ the conical crack points for which $\omega'_{c_j}\neq \emptyset$.

Now, if $x$ is a smooth boundary point, then the
intersection of a small neighborhood of $x$ and $\pa\Omega$ is a
smooth curve $\gamma$. We then have two possibilities: either
$\Omega$ lies on one side of $\gamma$, or $\Omega$
lies in both sides of $\gamma$ (the latter corresponds to a {smooth}
crack point). Hence, we have one inward normal unit vector in the first case, and two inward normal
unit vectors in the second case. 

The most important idea behind the construction of the unfolded domain
is that each smooth crack point $x$ should be covered by two points,
which correspond to the two sides of the crack and the two possible
non-tangential limits at $x$ of functions defined on $\Omega$. (See
\cite{LiMN, MN} for details.) The point is to cover each smooth crack
curve by two parallel smooth curves, so that we distinguish the two
normal directions from which we approach the boundary. When cracks
ramify, we need further to distinguish from which side we approach the
point of ramification.  In general, if a crack point has ramification $k$, then we cover this point by $k$ points. 

Following \cite{MN}, we define the \emph{unfolded boundary}
$\pa^u\Omega$ as the set of inward pointing unit normal vectors to the
smooth part $\Omega_0$ of $\pa\Omega$.  We have a canonical projection
$\kappa : \pa^u\Omega \to \Omega_0$, which is two-to-one over smooth
crack points and one-to-one over the non-crack smooth points. We then
define the \emph{unfolded domain} as 
\begin{equation}\label{unfoldeddomain}
\Omega^u:=\Omega \cup \pa^u \Omega.
\end{equation}
We can think of $\Omega^u$ as a (generalized) polygonal domain,
without cracks. (The angles at some points may be $2\pi$, i.e., some
two adjacent edges may be parallel.) The boundary of $\Omega^u$ is
given by $\pa^u\Omega$ together with the (true) non-crack vertices of
$\Omega$, together with  $k$-covers of crack points with ramification
$k$. Note that the smooth part of $\pa(\Omega^u)$ is given by
$\pa^u\Omega$, that is, by the union of the smooth part of $\pa\Omega$
and the $2$-covers of the smooth crack curves.

 We will let, from now on, $\cC^u:=\{c_{ji}\;|\; c_{ji} \text{
  covers } c_j \in \cC, i=1, \cdots, k_{c_j}\}\subset \pa(\Omega^u)$, where $k_{c_j}$ is the number of 
  connected components of
$\omega_{c_j}$, if $c_j$ is not a conical crack point (that is, the ramification number), or the number of connected components of $\omega_{c_j}''$, if $c_j$ is a conical crack
point. In this case,  if $\omega_{c_j}'\neq \emptyset$, we regard the
cover over $c_j$ as a non-crack point $c_{j0}$ together with a
$k_{c_j}$-cover.

With the notation as above, we see that the set of (true) vertices of $\Omega^u$ is given by
\begin{equation}
\label{vertex}
V^u= \Omega^{(0)}\cup \cC^u \cup \{c_{j0}\}_{j=1, \cdots, m'}.
\end{equation}
 Moreover, we have that
$\Omega^u\setminus V^u$ is given by a disjoint union of smooth curves,
as in the no-cracks case.

For $x\in V^u$, we denote by $\omega_x^u$ the basis of the cone at $x$ in $\Omega^u$. If $c\in\cC\subset \pa\Omega $ is a singular outer or inner crack point and the
basis of the cone at $c$ (in $\Omega$) is written as a union of open intervals
$\omega_c\cong \cup_{h=1}^{k_c} I_h$, 
then for each $c_h\in \cC^u\subset \pa(\Omega^u)$ in a $k_c$-cover of $c$, we have
$\omega^u_{c_h}= I_h$ is connected, and there is a neighborhood
$V_{c_h}$ such that $V_{c_h}\cap \Omega^u\cong (0,1)\times
\omega^u_{c_h}$. If $c$ is a conical crack point, then the same as above happens, replacing $\omega_c$ by the crack part $\omega''_c$. At the 'non-crack' point $c_0$, we have
$\omega^u_{c_{0}}=\omega_c'$ (possibly disconnected). 

Now we desingularize $\Omega^u$ and follow the same construction as in
Subsection \ref{highdimension} for $\Omega^u$. We obtain a groupoid
$\cG^u$ with the unit space isomorphic to $\partial' \Sigma(\Omega^u)$,
i.e., to a desingularization of $\pa \Omega^u$. Denote by $\cH= [0,1)\rtimes (0,1)$ the transformation
  groupoid. {F}rom (\ref{grpd.nocrack}), we define the \emph{boundary
    groupoid associated to a conical domain with ramified cracks} as
\begin{eqnarray*}
  \cG^u & := & \left(
  \coprod\limits_{x\in V^u}\cH\times (\partial\omega_{x}^u)^2\right) \quad \bigcup\limits_{\varphi} \quad (\pa^u \Omega)^2\\
  & = & \left(
  \coprod\limits_{p_i \in \Omega^{(0)}}\cH\times (\partial\omega_{p_i})^2
  \quad \bigcup \quad
     \coprod\limits_{ j=1}^{m' }\cH\times (\partial\omega'_{c_j})^2
   \quad \bigcup 
   \coprod\limits_{c_{jh}\in\cC^u}  \cH \times (\partial I_{c_j}^h)^2 
  \right) \quad\\
    & & \quad \bigcup\limits_{\varphi} \quad (\pa^u \Omega)^2,
\end{eqnarray*}
where $ I_{c_j}^h$ is the $h$-th connected component of $\omega_{c_j}$, respectively, of $\omega''_{c_j}$, if $c_j$ is non-conical, respectively, $c_j$ is a conical crack point, and $\varphi=(\varphi_{x})_{x\in V^u}$. Noting that $\pa \omega_x$ is a discrete set, with $\omega_x$ connected on covers of crack points, and denoting by $\cP_k$ the pair groupoid of a discrete set with $k$ elements, we have (writing $(\coprod A)^\alpha$ for the disjoint union of $\alpha$ copies of $A$):
\begin{equation}
\label{grpd.crack}
 \cG^u = \left(
  \coprod\limits_{p_i \in \Omega^{(0)}}\cH\times \cP_{2k_{p_i}}    \; \bigcup \;
    \coprod\limits_{ j=1, \cdots, m' }\cH\times \cP_{2k_{c_j}'} \; \bigcup \; 
  \left(\coprod\limits  \cH \times \cP_{2} \right)^\alpha \right) \;  \bigcup\limits_{\varphi} \; (\pa^u \Omega)^2,
\end{equation} 
where $k_{c_j}'$ is the number of connected components of $\omega_{c_j}'$ and $\alpha:= k_{c_1} + \cdots + k_{c_m}$ is a total ramification number of $\Omega$. The space of units of $\cG^u$ is given by
\begin{eqnarray}
\label{units.crack}
&&\\
  M &:=& \left(\coprod\limits_{p_i \in
    \Omega^{(0)}} [0,1)\times \partial\omega_{p_i} 
      \;\bigcup \;
    \coprod\limits_{ j=1}^{m' }\, [0,1)\times \partial\omega'_{c_j}
    \; \bigcup \;
    \coprod\limits_{c_{ji}\in\cC^u}  [0,1)\times \partial\omega^u_{c_{ji}} 
    \right)\;
    \bigcup\limits_{\varphi} \pa^u \Omega \nonumber\\
    & \cong &  \left(\coprod\limits_{ i=1}^{l} \left(\coprod[0,1)\right)^{2k_{p_i} } \; \bigcup \coprod\limits_{j=1}^{m'}\left(\coprod [0,1)\right)^{2k_{c_j}'}\; \bigcup \;
   \left( \coprod\limits  [0,1)\right)^{2\alpha}  \right)\;
    \bigcup\limits_{\varphi} \; \;\pa^u \Omega, \nonumber
\end{eqnarray}
which is diffeomorphic to a desingularization of $\pa(\Omega^u)$, as
in Section \ref{weightedspaces}. The unfolded boundary $\pa^u\Omega$
is an open dense set of $M$ and
\begin{equation*}
\label{ }
\pa M= \bigcup\limits_{x \in V^u} \{ x \} \times \pa \omega_x^u =
\bigcup\limits_{p \in \Omega^{(0)}} \{ p \} \times \pa \omega_p \;
\bigcup\limits_{c\in \cC^u} \; \{ c \} \times \pa \omega_c^u \;
\bigcup\limits_{j=1}^{m'} \; \{ c_j \} \times \pa \omega_{c_j}'.
\end{equation*}

\begin{definition}\label{def.layerC*crack}
The \emph{layer potentials $C^*$-algebra} associated to a polygonal domain $\Omega$ with ramified cracks  is defined as the groupoid $C^*$-algebra $C^*(\cG^u)$, where $\cG^u$ is the unfolded boundary groupoid as in (\ref{grpd.crack}).
\end{definition}

We summarize below the main properties of $\cG^u$ and its
$C^*$-algebra (see Propositions \ref{Groupoid} and \ref{C*alg}). As above, we let
$\Omega^{(0)}=\{p_1,p_2,\cdots, p_l\}$ be the set of conical,
non-crack points of $\Omega$, $\cC=\{c_1, \cdots, c_m\}$ be the set of
singular crack points, with $c_1, \cdots, c_{m'}$ the conical crack
points with not empty no-cracks part, and $\cC^u:=\{c_{ji}\;|\; c_{ji}
\text{ covers } c_j \in \cC, i=1, \cdots, k_{c_j}\}$.

\begin{proposition}\label{prop.grpd.crack}
Let $\cG^u$ be the groupoid (\ref{grpd.crack}) associated to a
polygonal domain with cracks $\Omega\subset \R^2$. Denote by $\pa^u
\Omega$ the unfolded boundary.  Then, $\cG^u$ is a Lie groupoid with
units $M$ such that
\begin{enumerate}

\item   $M_0=int(M) =\pa^u \Omega$ is an invariant subset and
$\cG^u_{M_0} \cong M_0 \times M_0.$

\item  $ \{ x \} \times \pa
\omega_x^u\subset \pa M$ is an invariant subset, for $x\in V^u$, and $ \pa  M$ is an invariant subset with
\begin{eqnarray*}
\cG^u_{\pa M}
&=& \coprod\limits_{i=1}^{l} (\pa\omega_{p_i})^2  \times (\R^+ \times \{p_i\}) \; \bigcup \; 
\coprod\limits_{j=1}^{m'} (\pa\omega'_{c_j})^2  \times (\R^+ \times \{c_j\}) \; \\ 
&&\bigcup \; \coprod\limits_{c_{jh}\in\cC^u}  (\pa I_{c_j}^h)^2 \times (\R^+ \times \{c_{jh}\}) 
\end{eqnarray*} 
where $ I_{c_j}^h$ is the $h$-th connected component of $\omega_{c_j}$, if $j>m'$, or of $\omega''_{c_j}$, if $m\leq m'$.

\item $A(\cG^u)\cong {^bTM}$, the $b$-tangent bundle of $M$.

\item $\cG^u$ is amenable, i.e., $C^*(\cG^u)\cong C^*_r(\cG^u) $.

\item We have the following exact sequence
 
\begin{equation*}
\xymatrix{ 
0 \ar[r]
  & \mathcal{K}\ar[r]
      &C^*(\cG^u) \ar[r]
          & \cA \ar[r] 
              &0,
   }
\end{equation*}       
  with $\cA:=\left(\bigoplus\limits_{i=1}^{l} M_{k_{p_i}}(\cC_0(\mathbb{R}^+))\right) \oplus \left( \;\bigoplus\limits_{j=1}^{m'} M_{k'_{c_j}}(\cC_0(\mathbb{R}^+))\right)\oplus \left( \bigoplus\limits_{\alpha \text{ times}}M_2(\cC_0(\mathbb{R}^+)\right) $, where $k_{p_i}$, $k'_{c_{j}}$ denote the number of elements of $\pa \omega_{p_i}$ and $\pa \omega_{c_j}'$, respectively, and $\alpha = k_{c_1}+\cdots k_{c_m}$ is the total ramification number of $\Omega$ (see Equation \ref{grpd.crack}).
\end{enumerate}
\end{proposition}

We also have an analogue of Theorem \ref{thm.equiv} for polygonal domains with ramified cracks. Recall from (\ref{vertex}) that the number of true conical points of $\Omega^u$ is given by $\sharp (V^u) = l+m'+\alpha$ with $l$ the number of true conical
non-crack points of $\Omega$, $m'$ the number of conical crack points
with a not empty no-crack part and $\alpha:= k_1 + \cdots k_m$ can be regarded as the
total ramification number of $\Omega$ (with $m$ is the number of all
singular crack points). We then have:

\begin{proposition}\label{cor.equiv.crack}
Let $\Omega_1, \, \Omega_2\subset \R^2$ be two polygonal domains,
possibly with cracks, and $\cG^u_1$, $\cG^u_2$ be the unfolded
boundary groupoids as in (\ref{grpd.crack}). If, using the notation
above, $l_1+m'_1+\alpha_1= l_2+m'_2+\alpha_2$, then $\cG_1^u$ is
equivalent to $\cG_2^u$ and $C^*(\cG_1^u )$ is Morita equivalent to
$C^*(\cG^u_2)$
\end{proposition}


\smallskip

\section{$K$-theory of the layer potentials $C^*$-algebra}\label{Ktheory}

In this section, we compute the $K$-theory groups of the  layer potentials $C^*$-algebra of the boundary groupoid of a conical domain and of the associated
indicial algebras. We obtain, in particular,  that the groupoid $K$-theory is the same for domains with the same number of singularities. We refer to \cite{ConnesBook, RLL} for the general concepts and results regarding $K$-theory for $C^*$-algebras. See also \cite{LeGallMonth, MelroseN, Monthubert, Moroianu, Nicola} for some computations of the $K$-groups of $C^*$-algebras associated to singular domains.

We consider first the case of a straight cone $\Omega=\R^+\omega
\subset \R^n$, with $\omega\subset S^{n-1}$ and let $\cJ= \cH\times
(\pa \omega)^2$ as in Section \ref{straightcone}, where
$\cH=[0,\infty)\rtimes \R^+$. Then, from Lemma \ref{lem.C*.cJ}, if
  $n\geq 3$,
$$
  C^*(\cJ)\cong C^*(\cH)\otimes \cK\cong
  \left(\cC_0([0,\infty))\rtimes \R^+\right)\otimes \cK.
$$ 
and, if $n=2$, 
$$C^*(\cJ)\cong C^*(\cH)\otimes M_k(\C)=M_k(C^*(\cH)),$$
where $k$ is the number of elements of $\pa \omega$. Since
$K_j(C^*(\cH)\otimes \cK)=K_j(M_k(C^*(\cH)))=K_j(C^*(\cH))$, $j=0,1$,
and since, by the Connes' Thom isomorphism \cite{ConnesBook}
$$ 
     K_j(\cC_0([0,\infty))\rtimes \R^+) = K_{1-j}(\cC_0([0,\infty)))=0,
    \quad j=0,1
$$ 
(see \cite{RLL}) we get that
\begin{equation}
\label{Kcone}
K_*(C^*(\cJ))=0
\end{equation}
for the layer potentials $C^*$-algebra of a straight cone.

Now we consider the general case. We start with computing the $K$-groups of the
 indicial algebra, i.e., at the boundary.  We assume first that
$n \geq 3$. As in Definition \ref{domain}, we denote by $\Omega\subset
\R^n$ a domain with distinct conical points $\{p_1, p_2, \cdots,
p_l\}$, with no cracks. Also denote by $\omega_i\subset S^{n-1}$ the
basis of the cone that corresponds to $p_i$. (Note that $\omega_i$ may
be disconnected).  Following the notations in Section
\ref{GpdConstruction}, we have
\begin{equation*}
\cG|_{\pa M}= \coprod\limits_{i=1}^{l} (\pa\omega_i \times
\pa\omega_i) \times (\R^+ \times \{p_i\})
\end{equation*} 
with unit space $\pa M$. Since $C^*((\pa\omega_i)^2 ) = \cK$ and
$C^*(\R^+ \times \{p_i\})= \cC_0(\R)$, we obtain
\begin{equation*}
C^*(\cG|_{\pa M})=  \bigoplus\limits_{i=1}^{l} \cK \otimes \cC_0(\R).
\end{equation*}

Consider next the case $n=2$, where we allow $\Omega$ to have
cracks. Let $V^u$ be the set of vertices of the unfolded domain
$\Omega^u$, as in the previous section (so $x\in V^u$ if $x$ is a
non-crack conical point of $\Omega$ or $x$ is in a cover of a singular
crack point). Denote by $\omega_x^u$ the basis of the cone that
corresponds to a vertex $x$ in $\Omega^u$. In this case,
$\pa\omega_x^u$ is a discrete set. Let $k_x$ be the number of points
in $\pa\omega_x^u$ ($k_x$ is even.) We have $C^*(\pa\omega_x^u \times
\pa\omega_x^u)= C^*(\cP_{k_x}) = M_{k_x}(\C). $ Hence, we obtain
\begin{equation*}
C^*(\cG_{\pa M})= \bigoplus\limits_{x\in V^u} M_{k_x}(\C) \otimes
\cC_0(\R) = \bigoplus\limits_{x\in V^u} M_{k_x}(\cC_0(\R)).
\end{equation*}
In any case, since $K_j \left(\cK \otimes \cC_0(\R)\right)=K_j
\left(M_k( \C)\otimes \cC_0(\R)\right)= K_j \left(\cC_0(\R) \right)$
for $j=1,2$, and $K_0\left(\cC_0(\R) \right)= 0$, $K_1\left(\cC_0(\R)
\right)=\Z$ (see \cite{RLL}), we have
\begin{lemma}\label{Kpa}
Let  $l$ denote the number of conical points in $\Omega$, $n\geq 3$ or in $\Omega^u$, for $n=2$ (in this case, including crack points). Then
$$K_0(C^*(\cG|_{\pa M}))=0 \,\,\,\, \text{and} \,\,\, K_1(C^*(\cG|_{\pa M}))=\bigoplus\limits_{i=1}^l \Z =\Z^l.$$
\end{lemma}

Since there is an interaction between connected
components at the groupoid level, the $K$-groups do not
depend on the (number of) connected components of $\omega_i$ and $\pa \omega_i$. On the other hand, it is clear from the computations above that the number of conical points
plays an essential role for the $K$-theory of $C^*(\cG|_{\pa M})$.  This fact should be expected from Theorem \ref{thm.equiv}, since Morita equivalent $C^*$-algebras have isomorphic $K$-groups. Hence:

\begin{corollary}\label{isom.K}
Let $\Omega,\, \Omega'\subset \R^n$ be two conical domains, and let
$\cG$, $\cG'$ be the boundary groupoids as in
(\ref{grpd.nocrack}). Assume $\Omega$ and $\Omega'$ have the same
number of true conical points. Then
$$K_j(C^*(\cG))\cong K_j(C^*(\cG')), \quad j=0,1.$$
\end{corollary}

A similar result holds also for domains with cracks, replacing $\Omega$ by $\Omega^u$, that is, 
`number of true conical points' by `$ l+m'+\alpha$' as in Corollary \ref{cor.equiv.crack}.

This result allows us to replace, at the level of $K$-theory, a
conical domain by any other with the same number of true vertices.

\begin{remark}\label{rmk.K.bcal}
The $K$-theory of the $b$-groupoid ${^b\cG}$ with units $M$ was
computed in \cite{MelroseN, Monthubert}, in the more general context
of manifolds with corners. If $M$ has a smooth boundary, one
gets $K_0(C^*({^b\cG}))=0, \quad K_1(C^*({^b\cG}))=\Z^{p-1},$ where
$p$ is the number of connected components of $\pa M$.  As we noted, if
$\Omega$ is a conical domain such that $\pa \omega_i$ connected for
all conical points (so $n\geq 3$), we have $\cG={^b\cG}$, where $\cG$
is the boundary groupoid associated to $\Omega$, and in that case the
number of connected components of $\pa M$ is exactly the number of
conical points.  We therefore obtain, for such $\Omega$,
\begin{equation}
\label{Kth}
K_0(C^*({\cG}))=0, \quad K_1(C^*({\cG}))=\Z^{l-1},
\end{equation}
where $l$ is the number of true conical points.  Now, given \emph{any}
conical domain $\Omega\subset \R^n$, for $n\geq 3$, with no cracks, we
can use Corollary \ref{isom.K} and the fact that one can always
construct $\Omega'\subset \R^n$ with the same (number of) true conical
points of $\Omega$ and such that the basis $\omega_{i}\subset S^{n-1}$
of the cone at the vertices has $\pa \omega_i$ connected, to conclude
that (\ref{Kth}) also holds in the general case, if the dimension is
greater than two.
\end{remark}

We saw that if $n\geq 3$, then the $K$-groups of the layer potential $C^*$-algebra can be computed from the $K$-groups of the $b$-groupoid $C^*$-algebra, using
Corollary \ref{isom.K}. We will give a simple, direct proof that
(\ref{Kth}) holds in general, that works as well for $n=2$, and does not use groupoid equivalence, nor the computations for the $b$-groupoid. The main point is to use functoriality to reduce to the case of a straight cone and make use of the fact that in this case the $K$-groups of the boundary groupoid are trivial.
 
\begin{theorem}\label{thm.K}
Let $\Omega\subset \R^n $ be a conical domain with no cracks and $\cG$
be the boundary groupoid as in (\ref{grpd.nocrack}). Then
$$K_0(C^*({\cG}))=0, \quad K_1(C^*({\cG}))=\Z^{l-1},$$
where $l$ is the number of true conical points of $\Omega$. 
\end{theorem}

\begin{proof}
By Proposition \ref{C*alg}, we have the following six-term exact sequence
$$\begin{CD} K_0(\cK) @>{i_*}>> K_0(C^*(\cG)) @>{q_*}>>
  K_0(C^*(\cG|_{\pa M})) \\ @A{\delta}AA & & @VV{}V
  \\ K_1(C^*(\cG|_{\pa M})) @<{q_*}<< K_1(C^*(\cG)) @<{i_*}<<
  K_1(\cK),
\end{CD}$$
where $i_*$ is induced by the inclusion and $q_*$ by the restriction
map, and $\delta$ is the index map. Therefore, since $K_0(\cK)= \Z$,
$K_1(\cK)=0$ (see \cite{RLL}), we obtain from Lemma \ref{Kpa} 
$$\begin{CD} \Z @>{i_*}>> K_0(C^*(\cG)) @>{q_*}>> 0 \\ @A{\delta}AA &
  & @VVV \\ \Z^l @<{q_*}<<
  K_1(C^*(\cG)) @<{i_*}<<0.
\end{CD}$$

We shall  show that $\delta$ is surjective, which yields $K_0(C^*(\cG))=0$ (since in that case $i_*$ is surjective and  the zero map), and that $\ker\delta \cong \Z^{l-1}$, which proves the result, as $K_1(C^*(\cG))\cong Im (q_*) = \ker\delta$.

Now let $V_l\cong (0,1)\omega_l$ be a conical neighborhood of the conical point $p_l\in \Omega^{(0)}$. Then $V_l$ is open in $\Omega$ and $\cG_{V_l}=\cH \times (\pa\omega_l)^2$ is an open subgroupoid   of $\cG=\cG_\Omega$, with units $[0,1)\times \pa\omega_l\subset M$ open as well, and $\pa ([0,1)\times \pa\omega_l)=\pa M \cap [0,1)\times \pa\omega_l$. Hence, there is an induced morphism of $C^*$-algebras $i_l: C^*(\cG_{V_l}) \to C^*(\cG)$, obtained by extending compactly supported functions on $\cG_{V_l}$ to $\cG$ by $0$. At the interior, that is, for pair groupoids, denoting by $M_1:=(0,1)\times \pa\omega_l$
we get a map 
$$k_l : C^*(M_1 \times M_1)\cong\cK(L^2(M_1)) \to C^*(M_0 \times M_0)\cong\cK(L^2(M_0)),$$ which is injective but not surjective, as its image is given by the operators with smooth kernel supported in $M_1\times M_1$. If we let $j_l: C^*((\pa\omega_l)^2\times \R^+) \to C^*(\cG_{\pa M})$ be the inclusion, 
we have now a commutative diagram, 
\[\xymatrix{
0\ar[r]&C^*(((0,1)\times \pa\omega_l)^2)=\cK\ar[r]\ar@{->}[d]^{k_l} & C^*(\cG_{V_l})\ar[r]\ar@{->}[d]^{i_l}&C^*((\pa\omega_l)^2\times \R^+)\ar[r]\ar@{->}[d]^{j_l} & 0\\
0\ar[r]&C^*(M_0 \times M_0)=\cK\ar[r]&C^*(\cG)\ar[r]&C^*(\cG_{\pa M})\ar[r]&0,} \]

By functoriality of the index map, it follows that we have moreover a commutative diagram
$$\begin{CD}
\Z = K_1(C^*((\pa\omega_l)^2\times \R^+))	@>{\delta'}>>  K_0(\cK(L^2(M_1))) =\Z \\
@V{ (j_l)_*}VV 		@V{ (k_l)_*}V \cong V \\
\Z^l =K_1(C^*(\cG_{\pa M})) @>\delta>> K_0(\cK(L^2(M_0))) = \mathbb{Z},
\end{CD}$$
where $(j_l)_*(z)=(0,.., 0,z)\in \Z^l$ and $\delta'$ is the index map for the straight cone $V_l$.  
We claim that $(k_l)_*:  K_0(\cK(L^2(M_1))= \mathbb{Z}\to \mathbb{Z}=K_0(\cK(L^2(M_0))$ is an isomorphism. This follows since for any $n\in \N$ we can find a projection in $\cK(L^2(M_0))$ with kernel compactly supported in $M_1$ and $\dim p(L^2(M_0))=\dim p(L^2(M_1))=n$, so $(k_l)_*$ is surjective. (Recall that the identification of the $K_0$-group of the compact operators with $\Z$ is such that the equivalence class of a compact projection gets mapped to the dimension of its image, see for instance \cite{RLL}.)

Now, as we saw in the beginning of this section, $K_*(\cG_{V_l})=0$ for a straight cone, and therefore the index map $\delta'$ is an isomorphism. Hence, $\delta\circ (j_l)_*$ is also an isomorphism, so it follows straightaway that $\delta$ is surjective and $K_0(C^*(\cG))=0$. 

For $K_1$, note that we have a split exact sequence
\begin{equation*}
\xymatrix{ 
0 \ar[r]
  &\ker \delta \ar[r]
      & K_1(C^*(\cG_{\pa M})) \ar[r]^\delta
          & K_0(\cK) \ar[r]
              &0,
   }
\end{equation*}   
where the splitting is given by $\phi: K_0(\cK)\to K_1(C^*(\cG_{\pa M}))$, $\phi:= { (j_l)_*} (\delta')^{-1} { (k_l)_*}^{-1}$, hence $\Z^l= K_1(C^*(\cG_{\pa M}))\cong \ker \delta \oplus K_0(\cK)$, so that $\ker\delta \cong \Z^{l-1}=K_1(C^*(\cG))$.

\end{proof}
 
It follows immediately that:

\begin{corollary}\label{thm.K}
Let $\Omega\subset \R^2 $ be a polygonal domain possibly with cracks
and $\cG^u$ be the unfolded boundary groupoid as in
(\ref{grpd.crack}). Then
$$K_0(C^*({\cG^u}))=0, \quad K_1(C^*({\cG^u}))=\Z^{l + m' + \alpha
  -1},$$ where $l$ is the number of true conical no-crack points of
$\Omega$, $m'$ is the number of conical crack points with a not empty
no-crack part and $\alpha$ is the total ramification number of
$\Omega$.
\end{corollary}

The group $K_1(C^*(\cG))$ is the recipient of the so-called
{analytic index} for operators on groupoids
\cite{MonthubertPierrot} given by the connecting map in the exact
sequence (\ref{ex.seq.symb}) associated to the principal symbol of
pseudodifferential operators on $\cG$
$$\delta: K_0(\cC_0(S^*A))=K^0(S^*A) \to K_1(C^*(\cG))=\Z^{l-1},$$
where $S^*A$ is the sphere Lie algebroid of $\cG$. If $P\in
\overline{\Psi^0(\cG)}$ is elliptic, then its symbol defines a class
$[\sigma(P)]\in K^0(S^*A)$ and the analytic index of $P$ is defined as
$an-ind(P):=\delta[\sigma(P)]$. (The analogue of $\delta$ in the
smooth case maps $K^0(S^*M)\to K_1(\cK)=\Z$ and computes the Fredholm
index.)

In applications, we are often more interested in studying the actual
Fredholmness of (a representation of) $P\in \Psi^\infty(\cG)$, which
is what we shall do next.

\smallskip

\section{Fredholm Criterion\label{Fredholmness}}

In this section, we obtain a Fredholm criterion for pseudodifferential
operators the boundary groupoid $\cG$ associated to a conical domain,
possibly with cracks in two dimensions, as in (\ref{grpd.nocrack}) and
(\ref{grpd.crack}).

Recall that $ \Psi^m(\cG)$ was defined in Section \ref{groupoid} as
the space of smooth, left-invariant, uniformly supported families
$(P_x)_{x\in M}$, with $P_x\in \Psi^m(\cG_x)$ and that there is a well-defined
symbol map $\sigma_0: \overline{\Psi^0(\cG)}\to \cC(S^*A)$, where
$A=A(\cG)$. The vector representation $\pi$ of $\Psi^{\infty}(\cG)$ on
$\cC^\infty_c(M)$ is uniquely determined by the equation
$(\pi(P)f)\circ r := P(f\circ r),$ where $f\in \cC^\infty_c(M)$ and
$P=(P_x)\in \Psi^m(\cG)$. Let also $\pi_x$, $x\in M$ be the regular
representation of $\Psi^{\infty}(\cG)$ on $\cC^\infty_c(\cG_x)$ with
$\pi_x(P)=P_x$.

To obtain our main results, we will use the spectral
properties of $\Psi^m(\cG)$ obtained in \cite{LN} (see also
\cite{LMN}). We first consider the  invariant set $\pa M$. It follows from the exact sequence
(\ref{ex.seq.inv}), using once more the fact
that the $C^*$-algebra of the pair groupoid is isomorphic to the compact operators, that we have \begin{equation}
\begin{CD}
  0 @>>> \cK @>\pi>> \overline{\Psi^0(\cG)}
  @>{\cR_{\pa M}\oplus  \sigma_0}>>
  \overline{\Psi^0(\cG_{\pa M})}\times_{\cC_0(S^*A_{\pa M})}\cC_0(S^*A)@>>> 0,
\end{CD}
\end{equation}
where $\cR_{\pa M}: \Psi^m(\cG) \to \Psi^m(\cG_{\pa M})$ is the
restriction map. Hence, $P\in \overline{\Psi^0(\cG)}$ has a representation as a Fredholm operator on $L^2(M)$ if,
and only if, $(\cR_{\pa M}\oplus \sigma_0)(P)$ is invertible, that is,
if $P$ is elliptic and the restriction to the indicial algebra
$\cR_{\pa M}(P)$ is invertible. It also follows that
$$\cK=\ker ( \cR_{\pa M}\oplus \sigma_0 )= (\ker \cR_{\pa M}) \cap
C^*(\cG).$$ Moreover, we have (see Theorem 9.3 in \cite{LN}, and also
Theorem 4 in \cite{LMN}):

\begin{theorem} \label{LNTheorem}
Let $\cG$ be a Lie groupoid with units $M$ and $M_0$ be an invariant
open dense subset of $M$ such that $\cG|_{M_0}\cong M_0\times
M_0$. Assume that the restriction of $\cG$ to $M\backslash M_0$ is
amenable, and that the vector representation $\pi$ is injective.  Let
$P\in \Psi^m(\cG)$ or $P\in \overline{\Psi^0(\cG_{\pa M})}$.
Then 
\begin{enumerate}
  \item $\pi(P):H^m(M)\rightarrow L^2(M)$ is Fredholm if, and only if,
    $P$ is elliptic and $\pi_x(P):= P_x :
    H^m(\mathcal{G}_x)\rightarrow L^2(\mathcal{G}_x)$ is invertible,
    for all $x\notin M_0$.
  \item  $\pi(P):H^m(M)\rightarrow L^2(M)$ is compact if, and
only if, $\sigma_0(P)=0$  and $\pi_x(P)=0$, for all $x\notin M_0$.
\end{enumerate}
\end{theorem}

We will use this theorem for the boundary groupoid $\cG$ with the space of
units $M$ as in (\ref{grpd.nocrack}) and for the unfolded groupoid $\cG^u$, on polygonal domains with cracks (\ref{grpd.crack}) . It follows from
Proposition \ref{grpd.nocrack} that $M_0 \simeq \Omega_0$ is an
invariant open dense subset of $M$ and $ \cG_{M_0}\cong M_0\times
M_0.$ Moreover $\cG_{\pa M}$ is amenable, since $\cG$ is (or
directly). We only need to show that the vector representation is
injective. This will stem from the fact that our boundary groupoid
$\cG$ is Hausdorff and that $M_0$ is an open dense subset. (Recall
that a space is Hausdorff if, and only if, if any two functions are
the same on a dense subset, then they coincide.)

\begin{lemma}
Let $\cG$ be a Hausdorff Lie groupoid with units $M$ and $M_0$ be an
invariant open dense subset of $M$ such that $\cG|_{M_0}\cong
M_0\times M_0$. Let $P\in \Psi^\infty(\cG)$. Then, if $P_x=0$ for
$x\in M_0$, then $P=0$ on $M$. In particular, the vector
representation $\pi$ is injective.
\end{lemma}

\begin{proof}
The first assertion is really Corollary 4.3 in \cite{Nistor1}. We
outline the argument: if $P_x=0$, for some $x\in M_0$, $P\in
\Psi^m(\cG)$ (or $P\in \overline{\Psi^0(\cG)}$, or $P\in C^*(\cG)$),
then by left-invariance $P_y=0$ for all $y\in M_0$, that is,
$P_{M_0}=0$. Since $P$ is a smooth family, $\cG$ is Hausdorff, and
$M_0$ is an open dense subset of $M$, we have $P=0$. In particular,
the regular representation $\pi_x$ is injective, for $x\in M_0$. But
in this case, the vector representation $\pi$ is equivalent to $\pi_x$
using the isometry $\cG_x\to M_0$. Hence, $\pi$ is also injective (can
also see this directly from the definition) and the second assertion
follows.
 \end{proof}

Now let $\Omega\subset \R^n$ be a conical domain with no cracks. Let  $\Sigma(\Omega)$ be the desingularization of $\Omega$ and $\partial'\Sigma(\Omega)\subset \pa\Sigma(\Omega)$ be the union of hyper-faces at infinity of $\Sigma(\Omega)$, corresponding to a desingularization of $\pa \Omega$, that is, $\partial'\Sigma(\Omega)=M$ as in (\ref{units.nocrack}).
Recall  that we have the identification (see Proposition
\ref{Identification})
$$\cK^{m}_{\frac{n-1}{2}}(\partial\Omega) \cong H^{m}(M)$$
for all $m\in\mathbb{Z}_+$, where $\cK^{m}_{\frac{n-1}{2}}(\partial\Omega)$ is the $m$-th weighted Sobolev space as in (\ref{def.weighted}) and $ H^m(M)=H^{m}(\partial'\Sigma(\Omega))$ is the Sobolev space defined using the induced Lie structure on $\pa \Omega$. The above also holds for domains with cracks, replacing $\Omega$ by the unfolded domain $\Omega^u$.

\begin{theorem}\label{fredholm}
Let $\cG$ be the groupoid (\ref{grpd.nocrack}) with units $M$ associated to a domain with conical points $\Omega\subset \R^n$ without cracks. Let $\Omega^{(0)}=\{p_1,p_2,\cdots, p_l\}$ be the set of conical points. Suppose $P = (P_x)_{x\in M}\in \Psi^m(\cG)$
or $P\in \overline{\Psi^0(\cG)}$.

 Then, $\pi(P):
\cK^{m}_{\frac{n-1}{2}}(\partial\Omega) \rightarrow \cK^{0}_{\frac{n-1}{2}}(\partial\Omega) $ is
Fredholm if, and only if, $P$ is elliptic and
\begin{equation*}
  P_x : H^{m}(\R^+ \times \pa \omega_p)\rightarrow L^2(  \R^+ \times \pa \omega_p)
\end{equation*}
is invertible, for any $x=(p, x'_p)\in \pa M$, with $p\in \Omega^{(0)}$, $x'_p\in \pa \omega_p$.
\end{theorem}

If $\Omega$ is the infinite straight cone with basis $\omega \subset S^{n-1}$, we have an identification $\cK^{m}_{\frac{n-1}{2}}(\pa\Omega)\cong H^{m}(\R^+ \times \pa \omega, g)$, and the metric $g= (r^{-1}dr)^2+(dx')^2$, where $x'$ are the coordinates on $S^{n-1}$. So in that case, we get for $x\in \pa \omega$,
\begin{equation*}
   P_x :\cK^{m}_{\frac{n-1}{2}}(\pa\Omega) \rightarrow
  \cK^{0}_{\frac{n-1}{2}}(\pa\Omega).
\end{equation*}

In case $\Omega$ is a polygonal domain with cracks, we can apply Theorem \ref{fredholm} to  the  groupoid (\ref{grpd.crack}) associated to the unfolded domain $\Omega^u$, as in Section \ref{crack}.  Let $\Omega^{(0)}=\{p_1,p_2,\cdots, p_l\}$  be the set of (true) conical points of $\Omega$ which are non-crack points, $\cC:=\{c_1, \cdots, c_m\}$ be the set of singular crack points, with $c_1, \cdots, c_{m'}$ the conical crack points with not empty no-cracks part, and $\cC^u:=\{c_{ji}\;|\; c_{ji} \text{ covers } c_j \in \cC, i=1, \cdots, k_{c_j}\}\subset \pa(\Omega^u)$ the set of covers of cracks. The set of vertices, i.e., conical points, of $\Omega^u$ is $V^u= \Omega^{(0)}\cup \cC^u \cup \{c_{j0}\}_{j=1, \cdots, m'}$.

\begin{corollary}\label{fredholm.crack}
Let $\cG^u$ be the groupoid (\ref{grpd.crack}) with units $M$ associated to a domain with conical points $\Omega\subset \R^n$ with cracks. 
Suppose $P = (P_x)_{x\in M}\in \Psi^m(\cG)$ or  $P\in \overline{\Psi^0(\cG)}$.
 Then, $\pi(P): \cK^{m}_{\frac{n-1}{2}}(\partial\Omega^u) \rightarrow \cK^{0}_{\frac{n-1}{2}}(\partial\Omega^u) $ is
Fredholm if, and only if, $P$ is elliptic and the following operators are invertible, for $x=(y, x'_y)\in \pa M$, with $y\in V^u$, $x'_y\in \pa\omega_y$:
\begin{itemize}
  \item  $P_x : H^{m}(\R^+ \times \pa \omega_p)\rightarrow L^2(  \R^+ \times \pa \omega_p)
$  with $y=p\in \Omega^{(0)}$;
  \item $P_x : H^{m}(\R^+ \times \pa \omega'_{c_j})\rightarrow L^2(  \R^+ \times \pa \omega'_{c_j})$, with $y= c_{j0}\in \cC^u$, $j=1,  \cdots, m'$;
  \item $P_x : H^{m}(\R^+ \times \pa I^h_{c_j})\rightarrow L^2(  \R^+ \times \pa I^h_{c_j})$
  with $y= c_{jh}\in \cC^u$, $j=1, 1, \cdots, m$, $h=1, \cdots k_{c_j}$ and $ I_{c_j}^h$ is the $h$-th connected component of $\omega_{c_j}$, respectively, of $\omega''_{c_j}$, if $c_j$ is non-conical, respectively, $c_j$ is a conical crack point.
\end{itemize}
\end{corollary}

Noting now that  $\pa  M= \bigcup\limits_{p \in \Omega^{(0)}} \{ p \} \times \pa
\omega_p$ can be written as a union of closed invariant subsets associated to each conical point $p_i$, we can also apply Theorem 7.4 in \cite{LN}. Let $\cR_{p_i}: \Psi^m(\cG)\to \Psi^m(\cG_{\{ p \} \times \pa\omega_p})=\Psi^m(\R^+ \times (\pa \omega_i)^2 \times \{p_i\})$. Recall that pseudodifferential operators with kernel in $\R^+ \times (\pa \omega_i)^2 \times \{p_i\}$ coincide with Mellin convolution operators on $\R^+\times\pa \omega_{p_i}$. If $Q\in \Psi^m(\R^+ \times (\pa \omega_i)^2 \times \{p_i\})$, we denote by $\widetilde{Q}: H^{m}(\R^+ \times \pa \omega_p)\rightarrow L^2( \R^+\times \pa \omega_p)$ the induced Mellin convolution operator.

\begin{theorem}\label{fredholm2}
Let $\cG$ be the groupoid (\ref{grpd.nocrack}) with units $M$ associated to a domain with conical points $\Omega\subset \R^n$ without cracks. Let $\Omega^{(0)}=\{p_1,p_2,\cdots, p_l\}$ be the set of conical points. Suppose $P = (P_x)_{x\in M}\in \Psi^m(\cG)$
or $P\in \overline{\Psi^0(\cG)}$.
 Then, $\pi(P): \cK^{m}_{\frac{n-1}{2}}(\partial\Omega) \rightarrow \cK^{0}_{\frac{n-1}{2}}(\partial\Omega) $ is
Fredholm if, and only if, $P$ is elliptic and $$\pi \circ \cR_{p_i}(P) : H^{m}(\pa \omega_{p_i})\rightarrow L^2( \pa \omega_{p_i})$$ is invertible, for any  $p_i\in \Omega^{(0)}$, or, equivalently, if the Mellin convolution operator $\widetilde{ \cR_{p_i}(P) }: H^{m}(\R^+ \times \pa \omega_p)\rightarrow L^2( \R^+\times \pa \omega_p)$ is an invertible operator. 
\end{theorem}

There is an analogue of the result above for domains with cracks, where we replace $\Omega$ by the unfolded domain $\Omega^u$.

We expect that  Theorems \ref{fredholm} and \ref{fredholm2} can be used to prove (or disprove)Fredolmness of certain integral operators arising from boundary value problems, namely in applications of the layer potentials method, generalizing the results in \cite{NQ1}. This goal will be pursued in a forthcoming paper.

\section*{Conclusion}

To a conical domain $\Omega$ we associate a boundary groupoid $\cG$ with space of units given by a desingularization $M$ of $\pa \Omega$. The layer potentials $C^*$-algebra associated to $\Omega$ is defined as the groupoid convolution algebra $C^*(\cG)$. For straight cones, this $C^*$-algebra is Morita equivalent to an algebra of Wiener-Hopf operators. In two dimensions, we allow  domains with ramified cracks, using the notion of unfolded domain. 

We study the structure of the boundary groupoid and its $C^*$-algebra, which is identified with an ideal  in the norm closure of order $0$ pseudodifferential operators on $\cG$. The invariant subsets of $\cG$ are given by  the smooth part $\Omega_0\subset \pa \Omega$, and we get also an invariant subset for each conical point $p$ given by $\{p\}\times \pa\omega_p$, with $\omega_p$ the basis of the local cone at $p$, so that
$$\cG_{\Omega_0}= \Omega_0\times \Omega_0, \quad \cG_{\{p\}\times \pa\omega_p}= (\pa\omega_p\times  \pa\omega_p)\times \R^+\times \{p\}.$$
It is seen that, up to equivalence, the boundary groupoid only depends on the number of conical points of $\Omega$, yielding Morita equivalent $C^*$-algebras. Moreover, we compute the $K$-groups of the layer potentials $C^*$-algebra, finding that
$$ K_0(C^*(\cG))=0, \quad K_1(C^*(\cG))=\Z^{l-1},$$
where $l$ is the number of (true) conical points of $\Omega$. 
As for pseudodifferential operators on $\cG$, at the interior we get a pseudodifferential operator on $\Omega_0\subset \pa\Omega$ smooth, and for each conical point $p$ we have a Mellin convolution operator $Q_p$ on $\R^+\times \pa\omega_p$.
Our final result is a Fredholm criterion, which yields that for $P\in \overline{\Psi^0(\cG)}$ or $P\in \Psi^\infty(\cG)$, with $P=(P_x)$, then the vector representation $\pi(P)$  on $C^\infty_c(M)$, 
$$\pi(P): \cK^{m}_{\frac{n-1}{2}}(\partial\Omega) \rightarrow \cK^{0}_{\frac{n-1}{2}}(\partial\Omega) $$
  is a Fredholm operator between suitable weighted Sobolev spaces if, and only if, $P$ is elliptic and the following family of operators is invertible, for $x\in \pa M$,
\begin{equation*}
  P_x : H^{m}(\R^+ \times \pa \omega_p)\rightarrow L^2(  \R^+ \times \pa \omega_p).
\end{equation*}
Alternatively, we can replace the second condition above by invertibility of the Mellin convolution operators $Q_p$, for each conical point $p$. 

We expect our results to apply to the study of Fredholmness and compactness of boundary operators related to applications of the method of layer potentials for boundary value problems on conical domains.


\bibliographystyle{plain}
\bibliography{Reference}

\def\cprime{$'$} \def\cprime{$'$}
\begin{thebibliography}{10}

\bibitem{MonthubertSchrohe}
J.~Aastrup, S.~Melo, B.~Monthubert, and E.~Schrohe.
\newblock Boutet de {M}onvel's calculus and groupoids {I}.
\newblock {\em J. Noncommut. Geom.}, 4(3):313--329, 2010.

\bibitem{AJ2}
A.~Alldridge and T.~Johansen.
\newblock An index theorem for {W}iener-{H}opf operators.
\newblock {\em Adv. Math.}, 218(1):163--201, 2008.

\bibitem{ALNgeom}
B.~Ammann, R.~Lauter, and V.~Nistor.
\newblock On the geometry of {R}iemannian manifolds with a {L}ie structure at
  infinity.
\newblock {\em Int. J. Math. Math. Sci.}, 2004(1-4):161--193, 2004.

\bibitem{ALN}
B.~Ammann, R.~Lauter, and V.~Nistor.
\newblock Pseudodifferential operators on manifolds with a {L}ie structure at
  infinity.
\newblock {\em Ann. of Math. (2)}, 165(3):717--747, 2007.

\bibitem{BMNZ}
C.~B{\u{a}}cu{\c{t}}{\u{a}}, A.~Mazzucato, V.~Nistor, and L.~Zikatanov.
\newblock Interface and mixed boundary value problems on {$n$}-dimensional
  polyhedral domains.
\newblock {\em Doc. Math.}, 15:687--745, 2010.

\bibitem{CannasWeinstein}
A.~Cannas~da Silva and A.~Weinstein.
\newblock {\em Geometric models for noncommutative algebras}, volume~10 of {\em
  Berkeley Mathematics Lecture Notes}.
\newblock American Mathematical Society, Providence, RI, 1999.

\bibitem{ConnesBook}
A.~Connes.
\newblock {\em Noncommutative Geometry}.
\newblock Academic Press, New York London,, 1994.

\bibitem{CordesBVP}
H.~O. Cordes.
\newblock On the technique of comparison algebra for elliptic boundary problems
  on noncompact manifolds.
\newblock In {\em Operator theory: operator algebras and applications, {P}art 1
  ({D}urham, {NH}, 1988)}, volume~51 of {\em Proc. Sympos. Pure Math.}, pages
  113--130. Amer. Math. Soc., Providence, RI, 1990.

\bibitem{DebordLescure1}
C.~Debord and J.-M. Lescure.
\newblock {$K$}-duality for pseudomanifolds with an isolated singularity.
\newblock {\em C. R. Math. Acad. Sci. Paris}, 336(7):577--580, 2003.

\bibitem{DebordLescure2}
C.~Debord and J.-M. Lescure.
\newblock {$K$}-duality for pseudomanifolds with isolated singularities.
\newblock {\em J. Funct. Anal.}, 219(1):109--133, 2005.

\bibitem{DLN}
C.~Debord, J.-M. Lescure, and V.~Nistor.
\newblock Groupoids and an index theorem for conical pseudo-manifolds.
\newblock {\em J. Reine Angew. Math.}, 628:1--35, 2009.

\bibitem{EgorovSchulze}
Y.~Egorov and B.-W. Schulze.
\newblock {\em Pseudo-differential operators, singularities, applications},
  volume~93 of {\em Operator Theory: Advances and Applications}.
\newblock Birkh\"auser Verlag, Basel, 1997.

\bibitem{Els}
J.~Elschner.
\newblock The double layer potential operator over polyhedral domains. {I}.
  {S}olvability in weighted {S}obolev spaces.
\newblock {\em Appl. Anal.}, 45(1-4):117--134, 1992.

\bibitem{FJL}
E.~Fabes, M.~Jodeit, and J.~Lewis.
\newblock Double layer potentials for domains with corners and edges.
\newblock {\em Indiana Univ. Math. J.}, 26(1):95--114, 1977.

\bibitem{FJR}
E.~Fabes, M.~Jodeit, and N.~Rivi{\`e}re.
\newblock Potential techniques for boundary value problems on
  {$C^{1}$}-domains.
\newblock {\em Acta Math.}, 141(3-4):165--186, 1978.

\bibitem{Fol}
G.~Folland.
\newblock {\em Introduction to partial differential equations}.
\newblock Princeton University Press, Princeton, NJ, second edition, 1995.

\bibitem{Karoubi}
M.~Karoubi.
\newblock Homologie cyclique et {$K$}-th\'eorie.
\newblock {\em Ast\'erisque}, (149):147, 1987.

\bibitem{Kon}
V.~Kondrat{\cprime}ev.
\newblock Boundary value problems for elliptic equations in domains with
  conical or angular points.
\newblock {\em Trudy Moskov. Mat. Ob\v s\v c.}, 16:209--292, 1967.

\bibitem{Kress}
R.~Kress.
\newblock {\em Linear integral equations}, volume~82 of {\em Applied
  Mathematical Sciences}.
\newblock Springer-Verlag, New York, second edition, 1999.

\bibitem{LMN}
R.~Lauter, B.~Monthubert, and V.~Nistor.
\newblock Pseudodifferential analysis on continuous family groupoids.
\newblock {\em Doc. Math.}, 5:625--655 (electronic), 2000.

\bibitem{LN}
R.~Lauter and V.~Nistor.
\newblock Analysis of geometric operators on open manifolds: a groupoid
  approach.
\newblock In {\em Quantization of singular symplectic quotients}, volume 198 of
  {\em Progr. Math.}, pages 181--229. Birkh\"auser, Basel, 2001.

\bibitem{LeGallMonth}
P.Y Le~Gall and B.~Monthubert.
\newblock {$K$}-theory of the indicial algebra of a manifold with corners.
\newblock {\em {K}-Theory}, 23(2):105--113, 2001.

\bibitem{Lescure}
J.-M. Lescure.
\newblock Elliptic symbols, elliptic operators and {P}oincar\'e duality on
  conical pseudomanifolds.
\newblock {\em J. K-Theory}, 4(2):263--297, 2009.

\bibitem{Lew}
J.~Lewis.
\newblock Layer potentials for elastostatics and hydrostatics in curvilinear
  polygonal domains.
\newblock {\em Trans. Amer. Math. Soc.}, 320(1):53--76, 1990.

\bibitem{LP}
J.~Lewis and C.~Parenti.
\newblock Pseudodifferential operators of {M}ellin type.
\newblock {\em Comm. Partial Differential Equations}, 8(5):477--544, 1983.

\bibitem{LiMN}
H.~Li, A.~Mazzucato, and V.~Nistor.
\newblock Analysis of the finite element method for transmission/mixed boundary
  value problems on general polygonal domains.
\newblock {\em Electron. Trans. Numer. Anal.}, 37:41--69, 2010.

\bibitem{Mac}
K.~Mackenzie.
\newblock {\em Lie groupoids and {L}ie algebroids in differential geometry},
  volume 124 of {\em London Mathematical Society Lecture Note Series}.
\newblock Cambridge University Press, Cambridge, 1987.

\bibitem{Mazya}
V.~Maz{\cprime}ya.
\newblock Boundary integral equations.
\newblock In {\em Analysis, {IV}}, volume~27 of {\em Encyclopaedia Math. Sci.},
  pages 127--222. Springer, Berlin, 1991.

\bibitem{MN}
A.~Mazzucato and V.~Nistor.
\newblock Well-posedness and regularity for the elasticity equation with mixed
  boundary conditions on polyhedral domains and domains with cracks.
\newblock {\em Arch. Ration. Mech. Anal.}, 195(1):25--73, 2010.

\bibitem{MelroseAPS}
R.~Melrose.
\newblock {\em The {A}tiyah-{P}atodi-{S}inger index theorem}, volume~4 of {\em
  Research Notes in Mathematics}.
\newblock A K Peters Ltd., Wellesley, MA, 1993.

\bibitem{MelroseScattering}
R.~Melrose.
\newblock {\em Geometric scattering theory}.
\newblock Stanford Lectures. Cambridge University Press, Cambridge, 1995.

\bibitem{MelroseN}
R.~Melrose and V.~Nistor.
\newblock {$K$}-theory of {$C^*$}-algebras of $b$-pseudodifferential operators.
\newblock {\em Geom. Funct. Anal.}, 8:88--122, 1998.

\bibitem{IMitrea2}
D.~Mitrea and I.~Mitrea.
\newblock On the {B}esov regularity of conformal maps and layer potentials on
  nonsmooth domains.
\newblock {\em J. Funct. Anal.}, 201(2):380--429, 2003.

\bibitem{IMitrea3}
I.~Mitrea.
\newblock On the spectra of elastostatic and hydrostatic layer potentials on
  curvilinear polygons.
\newblock {\em J. Fourier Anal. Appl.}, 8(5):443--487, 2002.

\bibitem{MitreaNistor}
M.~Mitrea and V.~Nistor.
\newblock Boundary value problems and layer potentials on manifolds with
  cylindrical ends.
\newblock {\em Czechoslovak Math. J.}, 57(132)(4):1151--1197, 2007.

\bibitem{MM}
I.~Moerdijk and J.~Mr{\v{c}}un.
\newblock {\em Introduction to foliations and {L}ie groupoids}, volume~91 of
  {\em Cambridge Studies in Advanced Mathematics}.
\newblock Cambridge University Press, Cambridge, 2003.

\bibitem{Monthubert1}
B.~Monthubert.
\newblock Groupoids of manifolds with corners and index theory.
\newblock In {\em Groupoids in analysis, geometry, and physics ({B}oulder,
  {CO}, 1999)}, volume 282 of {\em Contemp. Math.}, pages 147--157. Amer. Math.
  Soc., Providence, RI, 2001.

\bibitem{Monthubert}
B.~Monthubert.
\newblock Groupoids and pseudodifferential calculus on manifolds with corners.
\newblock {\em J. Funct. Anal.}, 199:243--286, 2003.

\bibitem{MonthubertNistor}
B.~Monthubert and V.~Nistor.
\newblock A topological index theorem for manifolds with corners.
\newblock Preprint, 2011.

\bibitem{MonthubertPierrot}
B.~Monthubert and F.~Pierrot.
\newblock Indice analytique et groupo\"\i des de {L}ie.
\newblock {\em C. R. Acad. Sci. Paris S\'er. I Math.}, 325(2):193--198, 1997.

\bibitem{Moroianu}
S.~Moroianu.
\newblock {K}-theory of suspended pseudo-differential operators.
\newblock {\em {K}-Theory}, 28(2):167--181, 2003.

\bibitem{MRen}
P.~Muhly and J.~Renault.
\newblock C*-algebras of multivariate {W}iener-{H}opf operators.
\newblock {\em Trans. Amer. Math. Soc.}, 274 (1):1--44, 1982.

\bibitem{MRW}
P.~Muhly, J.~Renault, and D.~Williams.
\newblock Equivalence and isomorphism for groupoid {C}*-algebras.
\newblock {\em J. Operator Theory}, 17:3--22, 1987.

\bibitem{Nicola}
F.~Nicola.
\newblock {K}-theory of {SG}-pseudodifferential algebras.
\newblock {\em Proc. Amer. Math. Soc.}, 131(9):2841--2848, 2003.

\bibitem{Nistor1}
V.~Nistor.
\newblock Pseudodifferential operators on non-compact manifolds and analysis on
  polyedral domains.
\newblock In Booss B., G.~Grubb, and K.~P. Wojciechowski, editors, {\em
  Spectral geometry of manifolds with boundary and decomposition of manifolds},
  volume 366 of {\em Contemporary Mathematics}, pages 307--328, Rhode Island,
  2005. Amer. Math. Soc.

\bibitem{NQ1}
V.~Nistor and Y.~Qiao.
\newblock Single and double layer potentials on domains with conical points
  {I}: Straight cones.
\newblock Preprint, 2010.

\bibitem{NWX}
V.~Nistor, A.~Weinstein, and P.~Xu.
\newblock Pseudodifferential operators on differential groupoids.
\newblock {\em Pacific J. Math.}, 189(1):117--152, 1999.

\bibitem{Ren}
J.~Renault.
\newblock {\em A groupoid approach to {$C^{\ast} $}-algebras}, volume 793 of
  {\em Lecture Notes in Mathematics}.
\newblock Springer, Berlin, 1980.

\bibitem{RLL}
M.~R{\o}rdam, F.~Larsen, and N.~Laustsen.
\newblock {\em An introduction to {$K$}-theory for {$C^*$}-algebras}, volume~49
  of {\em London Mathematical Society Student Texts}.
\newblock Cambridge University Press, Cambridge, 2000.

\bibitem{SchroheSG}
E.~Schrohe.
\newblock Spaces of weighted symbols and weighted {S}obolev spaces on
  manifolds.
\newblock In {\em Pseudo-Differential Operators. Proceedings Oberwolfach 1986},
  volume 1256 of {\em Springer Lecture Notes in Mathematics}, pages 360--377,
  1887.

\bibitem{SchroheSchulze1}
E.~Schrohe and B.-W. Schulze.
\newblock Boundary value problems in {B}outet de {M}onvel's algebra for
  manifolds with conical singularities. {I}.
\newblock In {\em Pseudo-differential calculus and mathematical physics},
  volume~5 of {\em Math. Top.}, pages 97--209. Akademie Verlag, Berlin, 1994.

\bibitem{SchroheSchulze2}
E.~Schrohe and B.-W. Schulze.
\newblock Boundary value problems in {B}outet de {M}onvel's algebra for
  manifolds with conical singularities. {II}.
\newblock In {\em Boundary value problems, {S}chr\"odinger operators,
  deformation quantization}, volume~8 of {\em Math. Top.}, pages 70--205.
  Akademie Verlag, Berlin, 1995.

\bibitem{Schulze}
B.-W. Schulze.
\newblock {\em Boundary value problems and singular pseudo-differential
  operators}.
\newblock Pure and Applied Mathematics (New York). John Wiley \& Sons Ltd.,
  Chichester, 1998.

\end{thebibliography}

\end{document}